\documentclass[a4paper,11pt,fleqn]{article}
\usepackage[dvipsnames]{xcolor}
\usepackage{pstricks}
\usepackage{enumitem}
\usepackage{amsmath}
\usepackage{amssymb}
\usepackage{pifont}
\usepackage{theorem} 
\usepackage{euscript}
\usepackage{exscale,relsize}
\usepackage{epic,eepic}
\usepackage{multicol}
\usepackage{makeidx}
\usepackage{srcltx}         
\usepackage{url}
\usepackage{charter}
\usepackage{mathrsfs} 
\usepackage{graphicx}
\usepackage{authblk}
\newcommand{\email}[1]{\href{mailto:#1}{\nolinkurl{#1}}}
\usepackage[normalem]{ulem}     
\newlength{\mySubFigSize}
\definecolor{labelkey}{rgb}{0,0.08,0.45}
\definecolor{refkey}{rgb}{0,0.6,0.0}
\definecolor{Brown}{rgb}{0.45,0.0,0.05}
\definecolor{dgreen}{rgb}{0.00,0.49,0.00}
\definecolor{dblue}{rgb}{0,0.08,0.75}
\RequirePackage[dvips,colorlinks,hyperindex]{hyperref}
\hypersetup{linktocpage=true,citecolor=dblue,linkcolor=dgreen}
\newcommand{\Argmind}[2]{\ensuremath{\underset{\substack{{#1}}}%
{\text{Argmin}}\;\;#2}}
\renewcommand{\leq}{\ensuremath{\leqslant}}
\renewcommand{\geq}{\ensuremath{\geqslant}}
\newcommand{\minimize}[2]{\ensuremath{\underset{\substack{{#1}}}%
{\text{minimize}}\;\;#2 }}

\newcommand{\scal}[2]{{\langle{{#1}\mid{#2}}\rangle}}

\newcommand{\menge}[2]{\big\{{#1}~\big |~{#2}\big\}}

\newcommand{\HHH}{\ensuremath{\boldsymbol{\mathcal H}}}

\newcommand{\Argmin}{\ensuremath{\text{Argmin}}}
\newcommand{\HH}{\ensuremath{{\mathcal H}}}
\newcommand{\GG}{\ensuremath{{\mathcal G}}}
\newcommand{\KK}{\ensuremath{{\mathcal K}}}

\newcommand{\emp}{\ensuremath{{\varnothing}}}

\newcommand{\Id}{\ensuremath{\operatorname{Id}}\,}
\newcommand{\Idi}{\ensuremath{\operatorname{Id}_i}\,}

\newcommand{\RR}{\ensuremath{\mathbb{R}}}

\newcommand{\BL}{\ensuremath{\EuScript B}}
\newcommand{\PL}{\ensuremath{\EuScript P}}
\newcommand{\SL}{\ensuremath{\EuScript S}}
\newcommand{\KL}{\ensuremath{\EuScript K}}
\newcommand{\VL}{\ensuremath{\EuScript V}}
\newcommand{\TL}{\ensuremath{\EuScript T}}
\newcommand{\ML}{\ensuremath{\EuScript M}}
\newcommand{\JL}{\ensuremath{\EuScript J}}
\newcommand{\NL}{\ensuremath{\EuScript N}}
\newcommand{\FL}{\ensuremath{\EuScript F}}

\newcommand{\RPP}{\ensuremath{\left]0,+\infty\right[}}

\newcommand{\RX}{\ensuremath{\left]-\infty,+\infty\right]}}
\newcommand{\RXX}{\ensuremath{\left[-\infty,+\infty\right]}}

\newcommand{\NN}{\ensuremath{\mathbb N}}

\newcommand{\intdom}{\ensuremath{\text{int\,dom}\,}}

\newcommand{\exi}{\ensuremath{\exists\,}}

\newcommand{\ran}{\ensuremath{\text{\rm ran}\,}}

\newcommand{\zer}{\ensuremath{\text{\rm zer}\,}}
\newcommand{\pinf}{\ensuremath{{+\infty}}}
\newcommand{\minf}{\ensuremath{{-\infty}}}

\newcommand{\dom}{\ensuremath{\text{\rm dom}\,}}

\newcommand{\prox}{\ensuremath{\text{\rm prox}}}
\newcommand{\proj}{\ensuremath{\text{\rm proj}}}
\newcommand{\Fix}{\ensuremath{\text{\rm Fix}\,}}

\newcommand{\gra}{\ensuremath{\text{\rm gra}\,}}
\newcommand{\inte}{\ensuremath{\text{\rm int}}}

\newcommand{\sri}{\ensuremath{\text{\rm sri}\,}}

\newcommand{\infconv}{\ensuremath{\mbox{\small$\,\square\,$}}}
\newcommand{\pushfwd}{\ensuremath{\mbox{\Large$\,\triangleright\,$}}}
\newcommand{\zeroun}{\ensuremath{\left]0,1\right[}}   
\newcommand{\rzeroun}{\ensuremath{\left]0,1\right]}}

\newtheorem{theorem}{Theorem}[section]
\newtheorem{lemma}[theorem]{Lemma}
\newtheorem{corollary}[theorem]{Corollary}
\newtheorem{proposition}[theorem]{Proposition}

\theoremstyle{plain}{\theorembodyfont{\rmfamily}%
}
\theoremstyle{plain}{\theorembodyfont{\rmfamily}%
\newtheorem{example}[theorem]{Example}}
\theoremstyle{plain}{\theorembodyfont{\rmfamily}%
\newtheorem{remark}[theorem]{Remark}}
\theoremstyle{plain}{\theorembodyfont{\rmfamily}%
}
\theoremstyle{plain}{\theorembodyfont{\rmfamily}%
}
\theoremstyle{plain}{\theorembodyfont{\rmfamily}%
}
\theoremstyle{plain}{\theorembodyfont{\rmfamily}%
}
\theoremstyle{plain}{\theorembodyfont{\rmfamily}%
\newtheorem{problem}[theorem]{Problem}}

\numberwithin{equation}{section}
\begin{document}

\title{\sffamily\huge\vskip -13mm Monotone Operator Theory in Convex
Optimization\thanks{Contact 
author: P. L. Combettes, \email{plc@math.ncsu.edu},
phone:+1 (919) 515 2671. This work was supported by 
the National Science Foundation under grant CCF-1715671.}}

\author{Patrick L. Combettes\\
\small North Carolina State University,
Department of Mathematics,
Raleigh, NC 27695-8205, USA\\
}

\date{~}

\maketitle

\vskip -3mm

{\bfseries Abstract.} 
Several aspects of the interplay between monotone operator theory
and convex optimization are presented. The crucial role played by
monotone operators in the analysis and the numerical solution of
convex minimization problems is emphasized. We review the
properties of subdifferentials as maximally monotone operators
and, in tandem, investigate those of proximity operators as
resolvents. In particular, we study new transformations which map
proximity operators to proximity operators, and establish
connections with self-dual classes of firmly nonexpansive
operators. In addition, new insights and developments are 
proposed on the algorithmic front.

{\bfseries Keywords.}
Firmly nonexpansive operator,
monotone operator,
operator splitting, 
proximal algorithm,
proximity operator,
proximity-preserving transformation,
self-dual class,
subdifferential.

\section{Introduction and historical overview}
\label{sec:1}
In this paper, we examine various facets of the role of monotone
operator theory in convex optimization and of the interplay
between the two fields. Throughout, $\HH$ is a real Hilbert space
with scalar product $\scal{\cdot}{\cdot}$, associated norm
$\|\cdot\|$, and identity operator $\Id\!$. To put our
discussion in proper perspective, we first provide an
historical account and highlight some key results (see
Section~\ref{sec:2} for notation).

Monotone operator theory is a fertile area of nonlinear analysis
which emerged in 1960 in independent papers by 
Ka{\v{c}}urovski{\u\i}, Minty, and Zarantonello.  
Let $D$ be a nonempty
subset of $\HH$, let $A\colon D\to\HH$, and let $B\colon
D\to\HH$. Extending the ordering of functions on the
real line which results from the comparison of their increments, 
Zarantonello \cite{Zara60} declared $B$ is slower than $A$ if
\begin{equation}
\label{e:zarantonello1}
(\forall x\in D)(\forall y\in D)\quad
\scal{x-y}{Ax-Ay}\geq\scal{x-y}{Bx-By},
\end{equation}
which is denoted by $A\succcurlyeq B$. He then called $A$ 
(isotonically) monotone if $A\succcurlyeq 0$, that is,
\begin{equation}
\label{e:zarantonello5}
(\forall x\in D)(\forall y\in D)\quad
\scal{x-y}{Ax-Ay}\geq 0,
\end{equation}
and supra-unitary if $A\succcurlyeq\Id$. An instance of the 
latter notion can be found in \cite{Golo35}. In modern 
language, it corresponds to that of $1$-strong monotonicity.
An important result of \cite{Zara60} is the following.

\begin{theorem}[Zarantonello]
\label{t:zarantonello1}
Let $A\colon\HH\to \HH$ be monotone and Lipschitzian. Then 
$\ran(\Id+A)=\HH$.
\end{theorem}

Monotonicity captures two well-known concepts. First, if
$\HH=\RR$, a function $A\colon D\to\RR$ is monotone if and only
if it is increasing, that is,
\begin{equation}
\label{e:zarantonello2}
(\forall x\in D)(\forall y\in D)\quad x<y\quad\Rightarrow\quad
Ax\leq Ay.
\end{equation}
The second connection is with linear functional analysis: if 
$A\colon\HH\to\HH$ is linear and bounded, then it is monotone if
and only if it is positive, that is,
\begin{equation}
\label{e:zarantonello3}
(\forall x\in\HH)\quad \scal{x}{Ax}\geq 0. 
\end{equation}
In particular, if a bounded linear operator $A\colon\HH\to\HH$ is 
skew, that is $A^*=-A$, then it is monotone since 
\begin{equation}
\label{e:zarantonello7}
(\forall x\in\HH)\quad \scal{x}{Ax}= 0. 
\end{equation}
Regarding \eqref{e:zarantonello2}, a standard fact about
a differentiable convex function on an open interval of $\RR$ is
that its derivative is increasing. This property, which is 
already mentioned in Jensen's 1906 foundational paper
\cite{Jens06}, was extended in 1960 by
Ka{\v{c}}urovski{\u\i} \cite{Kach60}, who came up with the notion
of monotonicity \eqref{e:zarantonello5}, discussed strong
monotonicity, and observed that the
gradient of a differentiable convex function $f\colon\HH\to\RR$ 
is monotone (see also \cite{Vain59}).
In a paper submitted in 1960, Minty \cite{Mint62} also called 
$A\colon D\to\HH$ monotone if it satisfies
\eqref{e:zarantonello5}, and maximally monotone if it cannot be
extended to a strictly larger domain while preserving
\eqref{e:zarantonello5}. Although, strictly speaking, his
definitions dealt with single-valued operators, he established
results on monotone relations that naturally cover extensions
to what we now call set-valued operators. According to Browder
\cite{Brow65}, who initiated the study of set-valued
monotone operators in Banach spaces, the Hilbertian setting was
developed by Minty in unpublished notes. A set-valued operator
$A\colon\HH\to 2^{\HH}$ is maximally monotone if
\begin{equation} 
\label{e:maxmon2}
(\forall x\in\HH)(\forall u\in\HH)\quad\big[\:
(x,u)\in\gra A\quad\Leftrightarrow\quad(\forall (y,v)\in\gra A)\;\;
\scal{x-y}{u-v}\geq 0\;\big].
\end{equation}
In other words, $A$ is monotone and there exists no monotone 
operator $B\colon\HH\to 2^{\HH}$ distinct from $A$ such that
$\gra A\subset\gra B$. A key result of \cite{Mint62} is the
following theorem, which can be viewed as an extension of 
Theorem~\ref{t:zarantonello1} since a continuous monotone
operator $A\colon\HH\to\HH$ is maximally monotone.

\begin{theorem}[Minty]
\label{t:minty1}
Let $A\colon\HH\to 2^{\HH}$ be a monotone operator.
Then $A$ is maximally monotone if and only if $\ran(\Id+A)=\HH$.
\end{theorem}

The paper \cite{Mint62} also establishes an important connection
between monotonicity and nonexpansiveness, which we state in the
following form.

\begin{theorem}{\rm\cite[Prop.~4.4 and Cor.~23.9]{Livre1}}
\label{t:minty2}
Let $T\colon\HH\to\HH$. Then the following are equivalent:
\begin{enumerate}
\item
\label{t:minty2i}
$T$ is firmly nonexpansive, i.e., {\rm \cite{Brow67}}
\begin{equation}
\label{e:firm}
(\forall x\in\HH)(\forall y\in\HH)\quad\|Tx-Ty\|^2+
\|(\Id-T)x-(\Id-T)y\|^2\leq\|x-y\|^2.
\end{equation}
\item
\label{t:minty2ii}
$R=2T-\Id$ is nonexpansive, i.e.,
$(\forall x\in\HH)(\forall y\in\HH)\quad\|Rx-Ry\|\leq\|x-y\|$.
\item
There exists a maximally monotone operator 
$A\colon\HH\to 2^{\HH}$ such that $T$ is the resolvent of $A$,
i.e.,
\begin{equation}
\label{e:dresolvent}
T=J_A,\quad\text{where}\quad J_A=(\Id+A)^{-1}.
\end{equation}
\end{enumerate}
\end{theorem}

From the onset, monotone operator theory impacted areas such as
partial differential equations, evolution equations and
inclusions, and nonlinear equations; see for instance
\cite{Brez73,Bro63a,Bro63b,Ghou09,Kach68,Komu67,Lera65,Mint63,%
Show97,Vain72,Zara71,Zei90B}. In particular, in such problems, it
turned out to provide efficient tools to derive existence 
results. Standard references on the modern theory of monotone
operators are \cite{Livre1,Brez73,Phel93,Simo98}.
From a modeling standpoint, monotone operator theory
constitutes a powerful framework that reduces many problems in
nonlinear analysis to the simple formulation 
\begin{equation}
\label{e:zero}
\text{find}\;\;x\in\zer A=\menge{x\in\HH}{0\in Ax}=\Fix J_A,
\quad\text{where}\quad
A\colon\HH\to 2^{\HH}\;\text{is maximally monotone.}
\end{equation}
The most direct connection between monotone operator theory
and optimization is obtained through the subdifferential
of a proper function $f\colon\HH\to\RX$, i.e., the 
operator \cite{Mor63a,Mor63c,Rock63}
\begin{equation}
\label{e:subdiff}
\partial f\colon\HH\to 2^{\HH}\colon x\mapsto\menge{u\in\HH}
{(\forall y\in\HH)\;\;\scal{y-x}{u}+f(x)\leq f(y)}.
\end{equation}
This operator is easily seen to be monotone. In addition, 
from the standpoint of minimization, a straightforward yet 
fundamental consequence of \eqref{e:subdiff} is Fermat's rule.
It states that, for every proper function $f\colon\HH\to\RX$,
\begin{equation}
\label{e:fermat}
\Argmin f=\zer\partial f.
\end{equation}
The maximality of the subdifferential was first investigated by
Minty \cite{Mint64} for certain classes of convex functions, and
then by Moreau \cite{More65} in full generality. 

\begin{theorem}[Moreau]
\label{t:moreau1}
Let $f\colon\HH\to\RX$ be a proper lower semicontinuous convex
function. Then $\partial f$ is maximally monotone. 
\end{theorem}

One way to prove Moreau's theorem is to use
Theorem~\ref{t:minty1}; see \cite[Theorem~21.2]{Livre1} or
\cite[Exemple~2.3.4]{Brez73}. Interestingly, Moreau's proof in
\cite{More65} did not rely on Theorem~\ref{t:minty1} but on 
proximal calculus. The proximity operator of a function 
$f\in\Gamma_0(\HH)$ is \cite{Mor62b}
\begin{equation}
\label{e:dprox}
\prox_{f}\colon\HH\to\HH\colon
x\mapsto\underset{y\in\HH}{\text{argmin}}\;
\bigg(f(y)+\frac{1}{2}\|x-y\|^2\bigg).
\end{equation}
This operator is intimately linked to the subdifferential
operator. Indeed, let $f\in\Gamma_0(\HH)$. Then 
\begin{equation}
\label{e:moreau1}
(\forall x\in\HH)(\forall u\in\HH)\quad
u\in\partial f(x)\quad\Leftrightarrow\quad x=\prox_f(x+u).
\end{equation}
Alternatively,
\begin{equation}
\label{e:moreau2}
(\forall x\in\HH)(\forall p\in\HH)\quad
p=\prox_fx\quad\Leftrightarrow\quad x-p\in\partial f(p),
\end{equation}
which entails, using \eqref{e:subdiff}, that $\prox_f$ is
firmly nonexpansive. Furthermore, \eqref{e:fermat} and
\eqref{e:moreau2} imply that $\Argmin f=\Fix\prox_f$. Since fixed
points of firmly nonexpansive operators can be constructed by
successive approximations \cite{Brow67,Opia67}, a conceptual 
algorithm for finding a minimizer of $f$ is
\begin{equation}
\label{e:martinet1}
x_0\in\HH\quad\text{and}\quad(\forall n\in\NN)\quad
x_{n+1}=\prox_f x_n.
\end{equation}
This scheme was first studied by Martinet in the early 1970s
\cite{Mart70,Mart72}, and a special case in the context of
quadratic programming appeared in \cite[Sect.~5.8]{Bell66}. 
Though of limited practical use, this so-called proximal point
algorithm occupies nonetheless a central place in convex
minimization schemes because it embraces many fundamental
ideas and connections that have inspired much more efficient 
and broadly applicable minimization algorithms in the form 
of proximal splitting methods
\cite{Livre1,Banf11,Glow16}. The methodology underlying these
algorithms is to solve structured convex minimization problems
using only the proximity operators of the individual functions
present in the model. 

Moreau's motivations for introducing the proximity operator
\eqref{e:dprox} came from nonsmooth mechanics 
\cite{Mor63m,Mor64m,More66}. In recent years proximity operators 
have become prominent in convex optimization theory. For
instance, they play a central theoretical role in \cite{Livre1}.
On the application side, their increasing presence is
particularly manifest in the broad area of data 
processing, where they were introduced in \cite{Smms05}
and have since proven very effective in the modeling and the 
numerical solution of a vast array of problems in disciplines such 
as signal processing, image recovery, machine learning, and
computational statistics; see for instance
\cite{Beck09,Botr14,Byrn14,Chau13,Icip14,Jmaa18,Siop07,%
Banf11,Save18,Duch09,Xiao12,Jena11,Papa14,Vait18,Wrig12,Zeng14}.

At first glance, it may appear that the theory of subdifferentials
and proximity operators forms a self-contained corpus of
theoretical and algorithmic tools which is sufficient to deal
with convex optimization problems, and that the broader concepts
of monotone operators and resolvents play only a peripheral role
in such problems. A goal of this paper is to
show that monotone operator theory occupies a central position
in convex optimization, and that many advances in the latter would 
not have been possible without it. Conversely, we shall see
that some algorithmic developments in monotonicity methods have 
directly benefited from convex minimization methodologies. We
shall also examine certain aspects of the gap that separates the
two theories. Section~\ref{sec:2} covers notation and background.
Section~\ref{sec:3} studies subdifferentials as maximally monotone
operators and proximity operators as resolvents, discussing
characterizations, new proximity-preserving transformations, and 
self-dual classes. Section~\ref{sec:4} focuses on the use of
monotone operator theory in analyzing and solving convex
optimization problems, and it proposes new insights and 
developments. 

\section{Notation and background}
\label{sec:2}
We follow the notation of \cite{Livre1}, where one will find a
detailed account of the following notions.
The direct Hilbert sum of $\HH$ and a real Hilbert space $\GG$ is 
denoted by $\HH\oplus\GG$.
Let $A\colon\HH\to 2^{\HH}$ be a set-valued operator. We denote by 
$\gra A=\menge{(x,u)\in\HH\times\HH}{u\in Ax}$ the graph of $A$,
by $\dom A=\menge{x\in\HH}{Ax\neq\emp}$ the domain of $A$, by 
$\ran A=\menge{u\in\HH}{(\exi x\in\HH)\;u\in Ax}$ the range of $A$,
by $\zer A=\menge{x\in\HH}{0\in Ax}$ the set of zeros of $A$,
and by $A^{-1}$ the inverse of $A$, i.e., the set-valued operator 
with graph $\menge{(u,x)\in\HH\times\HH}{u\in Ax}$. 
The parallel sum of $A$ and 
$B\colon\HH\to 2^{\HH}$, and the parallel composition of $A$ by 
$L\in\BL(\HH,\GG)$ are, respectively,
\begin{equation}
\label{e:parasum}
A\infconv B=(A^{-1}+B^{-1})^{-1}\quad\text{and}\quad
L\pushfwd A=\big(L\circ A^{-1}\circ L^*)^{-1}.
\end{equation}
The resolvent of $A$ is $J_A=(\Id+A)^{-1}=A^{-1}\infconv\Id$. 
The set of fixed points of an operator $T\colon\HH\to\HH$
is $\Fix T=\menge{x\in\HH}{Tx=x}$.
The set of global minimizers of a function 
$f\colon\HH\to\RX$ is denoted by $\Argmin f$ and, if it is a
singleton, its unique element is denoted by $\text{argmin}\,f$.
We denote by $\Gamma_0(\HH)$ the class of lower semicontinuous 
convex functions $f\colon\HH\to\RX$ such that 
$\dom f=\menge{x\in\HH}{f(x)<\pinf}\neq\emp$. Now let
$f\in\Gamma_0(\HH)$. The conjugate of $f$ is the function 
$f^*\in\Gamma_0(\HH)$ defined by 
$f^*\colon u\mapsto\sup_{x\in\HH}(\scal{x}{u}-f(x))$.
The subdifferential of $f$ is defined in \eqref{e:subdiff}, and
its inverse is $(\partial f)^{-1}=\partial f^*$. The proximity 
operator $\prox_f$  of $f$ is defined in \eqref{e:dprox}.
We say that $f$ is $\nu$-strongly convex for some $\nu\in\RPP$
if $f-\nu\|\cdot\|^2/2$ is convex.
The infimal convolution of $f$ and $g\in\Gamma_0(\HH)$ is
\begin{equation}
\label{e:infconv1}
f\infconv g\colon\HH\to\RXX\colon x\mapsto
\inf_{y\in\HH}\big(f(y)+g(x-y)\big).
\end{equation}
Let $C$ be a convex subset of $\HH$. The interior of $C$ is denoted
by $\inte\,C$, the boundary of $C$ by $\text{bdry}\,C$, 
the indicator function of
$C$ by $\iota_C$, the distance function to $C$ by
$d_C$, the support function of $C$ by $\sigma_C$ and, 
if $C$ is nonempty and closed,
the projection operator onto $C$ by $\proj_C$, i.e.,
$\proj_C=\prox_{\iota_C}$. A point $x\in C$ is in
the strong relative interior of $C$, denoted by $\sri C$, if
the cone generated by $C-x$ is a closed vector subspace of $\HH$.
We define
\begin{itemize}
\item
$\BL(\HH,\GG)=\menge{T\colon\HH\to\GG}
{T\;\text{is linear and bounded}}$ and $\BL(\HH)=\BL(\HH,\HH)$.
\item
$\NL(\HH)=\menge{T\colon\HH\to\HH}{T\;\text{is nonexpansive}}$.
\item
$\FL(\HH)=\menge{T\colon\HH\to\HH}{T\;
\text{is firmly nonexpansive}}$.
\item
$\ML(\HH)=\menge{A\colon\HH\to 2^{\HH}}{A\;
\text{is maximally monotone}}$.
\item
$\SL(\HH)=\menge{A\colon\HH\to 2^{\HH}}
{(\exi f\in\Gamma_0(\HH))\;A=\partial f}$.
\item
$\JL(\HH)=\menge{T\colon\HH\to\HH}{(\exi A\in\ML(\HH))\;T=J_A}$.
\item
$\PL(\HH)=\menge{T\colon\HH\to\HH}{(\exi f\in\Gamma_0(\HH))\;
T=\prox_f}$.
\item
$\KL(\HH)=\menge{T\colon\HH\to\HH}{T=\proj_K\;\text{for some
nonempty closed convex cone}\;K\subset\HH}$.
\item
$\VL(\HH)=\menge{T\colon\HH\to\HH}{T=\proj_V\;\text{for some
closed vector space}\;V\subset\HH}$.
\end{itemize}
Facts mentioned in Section~\ref{sec:1} are summarized by the
inclusions
\begin{equation}
\label{e:incl}
\SL(\HH)\subset\ML(\HH),\;
\VL(\HH)\subset\KL(\HH)\subset\PL(\HH)\subset\JL(\HH)
=\FL(\HH)\subset\NL(\HH),\;
\text{and}\;\FL(\HH)\subset\ML(\HH).
\end{equation}
We have
\begin{equation}
\label{e:ja}
\big(\forall A\in\ML(\HH)\big)\quad A^{-1}\in\ML(\HH)
\quad\text{and}\quad
J_{A^{-1}}+J_A=A\infconv\Id+A^{-1}\infconv\Id=\Id.
\end{equation}

\begin{theorem}[Moreau's decomposition \cite{Mor62b,Mor63a,More65}]
\label{t:jjm0}
Let $f\in\Gamma_0(\HH)$ and set $q=\|\cdot\|^2/2$. Then 
$f\infconv q+f^*\infconv q=q$ and 
$\prox_f=J_{\partial f}=\nabla(f+q)^*=
\nabla(f^*\infconv q)=(\partial f^*)\infconv\Id=\Id-\prox_{f^*}$.
\end{theorem}
The next result brings together ideas from 
\cite{Bail77} and \cite{More65}.

\begin{theorem}[\rm\cite{Joca10}]
\label{Trio9-}
Let $h\colon\HH\to\RR$ be continuous and convex, and set 
$f=h^*-q$, where $q=\|\cdot\|^2/2$. 
Then the following are equivalent:
\begin{enumerate}
\item
\label{Trio9-i}
$h$ is Fr\'echet differentiable on $\HH$ and 
$\nabla h\in\NL(\HH)$.
\item
\label{Trio9-v}
$h$ is Fr\'echet differentiable on $\HH$ and 
$\nabla h\in\FL(\HH)$.
\item
\label{Trio9-vi}
$q-h$ is convex.
\item
\label{Trio9-vii}
$h^*-q$ is convex.
\item
\label{Trio9-viii}
$f\in\Gamma_0(\HH)$ and
$h=f^*\infconv q=q-f\infconv q$.
\item
\label{Trio9-ix}
$f\in\Gamma_0(\HH)$ and
$\prox_{f}=\nabla h=\Id-\prox_{f^*}$.
\end{enumerate}
\end{theorem}

\begin{lemma}{\rm\cite[Prop.~2.58]{Livre1}}
\label{l:1}
Let $f\colon\HH\to\RR$ be G\^ateaux differentiable, let 
$L\in\BL(\HH)$, and suppose that $\nabla f=L$. 
Then $L=L^*$, $f\colon x\mapsto f(0)+(1/2)\scal{Lx}{x}$, 
and $f$ is twice Fr\'echet differentiable.
\end{lemma}

\section{Subdifferentials as monotone operators}
\label{sec:3}
As seen in Section~\ref{sec:1}, from a convex optimization
perspective, the subdifferential and the proximity operators of 
a function in $\Gamma_0(\HH)$ constitute, respectively, prime 
examples of maximally monotone and firmly nonexpansive operators.
In this section with discuss some structural differences between
$\SL(\HH)$ and $\ML(\HH)$, and between $\PL(\HH)$ and $\JL(\HH)$.

\subsection{Characterization of subdifferentials}
\label{sec:31}
If $\HH=\RR$, then $\ML(\HH)=\SL(\HH)$; see
\cite[Sect.~24]{Rock70} or \cite[Cor.~22.23]{Livre1}. 
In general, however, this singular
situation no longer manifests itself. For instance if 
$0\neq A\in\BL(\HH)$ is skew (see \eqref{e:zarantonello7}), then 
$A\in\ML(\HH)$ since $A$ is monotone and continuous, 
but Lemma~\ref{l:1} asserts that it is not a gradient since it 
is not self-adjoint; it can therefore not be in $\SL(\HH)$.
A complete characterization of subdifferentials
as maximally monotone operators was given by Rockafellar in
\cite{Rock66}. An operator $A\colon\HH\to 2^\HH$ is 
cyclically monotone if, for every integer $n\geq 2$, every 
$(x_1,\ldots,x_{n+1})\in\HH^{n+1}$, and every 
$(u_1,\ldots,u_{n})\in\HH^{n}$,
\begin{equation} 
\label{e:mcyclic}
\big[\,(x_1,u_1)\in\gra A,\ldots,
(x_n,u_n)\in\gra A,\, x_{n+1}=x_1\,\big]
\quad\Rightarrow\quad 
\sum_{i=1}^n\scal{x_{i+1}-x_{i}}{u_{i}}\leq 0.
\end{equation}
In this case, $A$ is called maximally cyclically monotone if 
there exists no cyclically monotone operator 
$B\colon\HH\to 2^{\HH}$ such that $\gra B$ properly contains 
$\gra A$.

\begin{theorem}[Rockafellar] 
\label{t:mcyclic}
$\SL(\HH)=\menge{A\in\ML(\HH)}{A\;
\text{is maximally cyclically monotone}}$.
\end{theorem}

The question of representing a maximally monotone operator as the
sum of a subdifferential and a remainder component is a
challenging one. In the case of a monotone matrix $A$ such a
decomposition is obtained by writing $A$ as the sum of its
symmetric part (hence a gradient) and its antisymmetric part.
This observation motivated Asplund \cite{Aspl70} to investigate
the decomposition of $A\in\ML(\HH)$ as 
\begin{equation}
\label{e:asplund}
A=G+B,\;\,\text{where}\; 
G\in\SL(\HH)\;\text{and}\;
B\in\ML(\HH)\;\text{is acyclic.}
\end{equation}
Here, acyclic means that if $B=\partial g+C$ for some 
$g\in\Gamma_0(\HH)$ and some 
$C\in\ML(\HH)$, then $g$ is affine on $\dom A$. 
A sufficient condition for Alspund's cyclic+acyclic decomposition
\eqref{e:asplund} to exist for 
$A\in\ML(\HH)$ is that $\intdom A\neq\emp$ \cite{Borw06}.
Acyclic operators are not easy to apprehend, which shows that the
notion of a maximally monotone operator remains only partially
understood. A simpler decomposition was investigated in
\cite{Bor07a,Bor07b} by imposing that $B$ in \eqref{e:asplund} 
be the restriction of a skew operator in $\BL(\HH)$. Thus, the 
so-called Borwein-Wiersma decomposition of 
$A\in\ML(\HH)$ is 
\begin{equation}
\label{e:B-W}
A=G+B,\;\,\text{where}\; 
G\in\SL(\HH)\;\text{and}\;B=S|_{\dom B},\;\text{with}\;
S\in\BL(\HH)\;\text{and}\;S^*=-S.
\end{equation}
If $A\in\ML(\HH)$ and $\gra A$ is a vector subspace, then $A$
admits a Borwein-Wiersma decomposition if and only if $\dom
A\subset\dom A^*$ \cite[Thm.~5.1]{Baus10}. Another viewpoint on
the distinction between a general maximally monotone operator and a
subdifferential is presented in \cite{Zara88}.

\subsection{Characterizations of proximity operators}
\label{sec:32}

Exploring a different facet of the discussion of 
Section~\ref{sec:31}, we focus in this section on some properties
of the class of proximity operators as a subset of that of firmly
nonexpansive operators. We first review characterization
results and then study the closedness of
$\PL(\HH)$ under various transformations.

A first natural question that arises is how to characterize those
firmly nonexpansive operators which are
proximity operators. As mentioned in Section~\ref{sec:31}, on the
real line things are straightforward: since
$\ML(\RR)=\SL(\RR)$, Theorem~\ref{t:minty2} tells us
$\FL(\RR)=\PL(\RR)$. Alternatively, 
$T\colon\RR\to\RR$ belongs to $\PL(\RR)$ if and only if it is
nonexpansive and increasing \cite{Siop07}. In general, the
characterization of subdifferential operators given in
Theorem~\ref{t:mcyclic}, together with Theorem~\ref{t:minty2},
suggests introducing a cyclic version of \eqref{e:firm} to achieve
this goal. This leads to the following characterization.

\begin{proposition}{\rm\cite{Bart07}}
\label{p:2007}
Let $T\in\FL(\HH)$. Then $T\in\PL(\HH)$ if and only if, for every
integer $n\geq 2$ and every $(x_1,\ldots,x_{n+1})\in\HH^{n+1}$
such that $x_{n+1}=x_1$, we have
$\sum_{i=1}^n\scal{x_i-Tx_i}{Tx_i-Tx_{i+1}}\geq 0$.
\end{proposition}

For our purposes, a more readily exploitable characterization is
the following result due to Moreau (see also 
Theorem~\ref{Trio9-}).

\begin{theorem}{\rm\cite{More65}}
\label{t:jjm1}
Let $T\in\NL(\HH)$ and let
$q=\|\cdot\|^2/2$. Then $T\in\PL(\HH)$ if and 
only if there exists a differentiable convex function 
$h\colon\HH\to\RR$ such that $T=\nabla h$. 
In this case, $T=\prox_{f}$, where $f=h^*-q$. 
\end{theorem}

\begin{corollary}{\rm\cite{More65}}
\label{c:Moreau1965!}
Let $T\in\BL(\HH)$ be such that $\|T\|\leq 1$. Then 
$T\in\PL(\HH)$ if and only if $T$ is positive and self-adjoint.
\end{corollary}

\subsection{Proximity-preserving transformations}
\label{sec:33}
A transformation which preserves firm nonexpansiveness may not be
proximity-preserving in the sense that it may not produce a 
proximity operator when applied to proximity operators. Here are
two examples.

\begin{example}[composition-based transformations]
\label{ex:20}
Transformations involving compositions are unlikely to be
proximity-preserving for a simple reason: in the linear case,
Corollary~\ref{c:Moreau1965!}
imposes that such a transformation preserve self-adjointness.
However, a product of symmetric matrices may not be symmetric.
A standard example is the Douglas-Rachford splitting operator 
$T_{A,B}$ associated with two operators $A$ and $B$ in $\ML(\HH)$ 
\cite{Livre1}, which will arise in \eqref{e:dr}. In general,
\begin{equation}
\label{e:dr7}
T_{A,B}=J_A\circ(2J_B-\Id)+\Id-J_B
=\frac{(2J_A-\Id)\circ(2J_B-\Id)+\Id}{2}
\in\JL(\HH)\smallsetminus\PL(\HH).
\end{equation}
The fact that $T_{A,B}\in\JL(\HH)$ follows from the equivalence
\ref{t:minty2i}$\Leftrightarrow$\ref{t:minty2ii} in 
Theorem~\ref{t:minty2}. On the other hand, examples when 
$T_{A,B}\notin\PL(\HH)$ for $A\in\SL(\HH)$ and $B\in\SL(\HH)$ 
can be easily constructed when $J_A$ and $J_B$ are $2\times 2$ 
matrices as explained above. In fact, when
$A$ and $B$ are linear relations in $\SL(\RR^N)$ 
($N\geq 2$), the genericity of \eqref{e:dr7} is established in 
\cite{Baus18}.
\end{example}

\begin{example}[Spingarn's partial inverse]
\label{ex:21}
Let $A\in\ML(\HH)$ and let $V$ be a closed vector subspace of
$\HH$. The partial inverse of $A$ with respect to $V$ is the 
operator $A_V\colon\HH\to 2^{\HH}$ with graph
\begin{equation}
\label{e:fiwejfweo}
\gra A_V=\menge{(\proj_Vx+\proj_{V^\bot}u,
\proj_Vu+\proj_{V^\bot}x)}{(x,u)\in\gra A}.
\end{equation}
This operator, which was introduced by Spingarn in 
\cite{Spin83}, can be regarded as an intermediate object between
$A$ and $A^{-1}$. As shown in \cite{Spin83}, 
$A\in\ML(\HH)$ $\Leftrightarrow$ $A_V\in\ML(\HH)$.
Therefore, by Theorem~\ref{t:minty2}, 
$A\in\ML(\HH)$ $\Leftrightarrow$ $J_{A_V}\in\JL(\HH)$. 
However,
\begin{equation}
\label{e:spingarn7}
A\in\SL(\HH)\quad\not\Rightarrow\quad J_{A_V}\in\PL(\HH).
\end{equation}
To see this suppose that $\HH=\RR^2$, let
$V=\menge{(\xi_1,\xi_2)\in\HH}{\xi_1=\xi_2}$, and let
$A=\partial f$, where
$f\colon(\xi_1,\xi_2)\mapsto\xi_1^2/2+\xi_1\xi_2+\xi_2^2$.
Then, for every $(\xi_1,\xi_2)\in\HH$,
$\partial f(\xi_1,\xi_2)=(\xi_1+\xi_2,\xi_1+2\xi_2)$, and
we obtain
\begin{equation}
\label{e:spingarn8}
A_V(\xi_1,\xi_2)=(2\xi_1+\xi_2,-\xi_1+2\xi_2)\quad\text{and}\quad
J_{A_V}(\xi_1,\xi_2)=\frac{1}{10}(3\xi_1-\xi_2,\xi_1+3\xi_2).
\end{equation}
Thus $J_{A_V}^*\neq J_{A_V}$ and Corollary~\ref{c:Moreau1965!} 
implies that $J_{A_V}\notin\PL(\HH)$.
\end{example}

Let us start with some simple proximity-preserving  
transformations.
\begin{proposition}
\label{p:jjm12}
Let $T\in\PL(\HH)$. Then the following hold:
\begin{enumerate}
\item
\label{p:jjm12i}
$\Id-T\in\PL(\HH)$.
\item
\label{p:jjm12ii}
Let $z\in\HH$. Then $z+T(\cdot-z)\in\PL(\HH)$.
\item
\label{p:jjm12iii}
$-T\circ(-\Id)\in\PL(\HH)$.
\item
\label{p:jjm12iv}
$J_T\in\PL(\HH)$. 
\end{enumerate}
\end{proposition}
\begin{proof}
Let $f\in\Gamma_0(\HH)$ be such that $T=\prox_f$, and set
$q=\|\cdot\|^2/2$.

\ref{p:jjm12i}: Set $g=f^*$. Then $g\in\Gamma_0(\HH)$ and 
Theorem~\ref{t:jjm0} states that $\prox_{g}=\Id-T$.

\ref{p:jjm12ii}: Set $g=f(\cdot-z)$. Then $g\in\Gamma_0(\HH)$ and 
$\prox_{g}=z+T(\cdot-z)$ \cite{Smms05}.

\ref{p:jjm12iii}: Set $g=f(-\cdot)$. Then $g\in\Gamma_0(\HH)$ and 
$\prox_{g}=-T\circ(-\Id)$ \cite{Smms05}.

\ref{p:jjm12iv}: Since $\PL(\HH)\subset\JL(\HH)\subset\ML(\HH)$, 
$J_T$ is well defined. Now set $g=f^*\infconv q$. Then 
$g\in\Gamma_0(\HH)$ and $g=(f+q)^*$. Thus, by
Theorem~\ref{t:jjm0},
$J_T=T^{-1}\infconv\Id=(\Id+\partial f)\infconv\Id
=\partial(f+q)\infconv\Id=(\partial g^*)\infconv\Id=\prox_g$.
\end{proof}

\begin{proposition}
\label{p:jjm24}
Let $f\in\Gamma_0(\HH)$, let $\GG$ be a real Hilbert space,
and let $M\in\BL(\HH,\GG)$ be such that 
$0\in\sri(\ran M^*-\dom f)$ and $MM^*-\Id_{\!\!\GG}$ is positive.
Set $q_{\HH}=\|\cdot\|_{\HH}^2/2$ and 
$q_{\GG}=\|\cdot\|_{\GG}^2/2$. Then 
$M\pushfwd\prox_{f}\in\PL(\GG)$. More specifically,
$M\pushfwd\prox_{f}=\prox_{\varphi}$,
where $\varphi=(f+q_{\HH})\circ M^*-q_{\GG}$.
\end{proposition}
\begin{proof}
Set $\varphi=(f+q)\circ M^*-q$.
The assumptions imply that $f\circ M^*\in\Gamma_0(\GG)$ and
that $q_{\HH}\circ M^*-q_{\GG}\colon\GG\to\RR$ is convex and 
continuous. Consequently, $\varphi\in\Gamma_0(\GG)$ and, using 
\eqref{e:parasum} and \cite[Cor.~16.53(i)]{Livre1}, 
\begin{align}
M\pushfwd\prox_f
&=\big(M\circ(\Id_{\!\!\HH}+\partial f)\circ M^*\big)^{-1}
\nonumber\\
&=\big(M\circ\partial(f+q_{\HH})\circ M^*\big)^{-1}\nonumber\\
&=\Big(\partial\big((f+q_{\HH})\circ M^*\big)\Big)^{-1}\nonumber\\
&=\Big(\Id_{\!\!\GG}+\partial\big((f+q_{\HH})\circ M^*
-q_{\GG}\big)\Big)^{-1}
\nonumber\\
&=\prox_{\varphi},
\end{align}
as claimed.
\end{proof}

We now describe a composite proximity-preserving transformation.

\begin{proposition}
\label{p:jjm2}
Let $I$ be a nonempty finite set and put $q=\|\cdot\|_{\HH}^2/2$.
For every $i\in I$, let $\omega_i\in\RPP$,
let $\GG_i$ be a real Hilbert space with identity operator
$\Idi$, put $q_i=\|\cdot\|_{\GG_i}^2/2$,
let $\KK_i$ be a real Hilbert space,
let $L_i\in\BL(\HH,\GG_i)\smallsetminus\{0\}$,
let $M_i\in\BL(\KK_i,\GG_i)\smallsetminus\{0\}$,
let $f_i\in\Gamma_0(\GG_i)$, 
let $g_i\in\Gamma_0(\GG_i)$, and let 
$h_i\in\Gamma_0(\KK_i)$. 
Suppose that $\sum_{i\in I}\omega_i\|L_i\|^2\leq 1$ and that, for
every $i\in I$, 
\begin{equation}
\label{e:jjm41}
0\in\sri\big(\dom h_i^*-M_i^*(\dom f_i\cap\dom g_i^*)\big)
\quad\text{and}\quad
0\in\sri(\dom f_i-\dom g_i^*).
\end{equation}
Set
\begin{equation}
\label{p:jjm1}
T=\sum_{i\in I}\omega_iL_i^*\circ\Big(\prox_{f_i}\infconv
\big(\partial g_i\infconv(M_i\pushfwd\partial h_i)\big)\Big)
\circ L_i.
\end{equation}
Then $T\in\PL(\HH)$. More specifically, 
\begin{equation}
\label{p:jjm3}
T=\prox_f,\quad\text{where}\quad f=
\Bigg(\sum_{i\in I}\omega_i\Big(\big(f_i+g_i^*+h_i^*\circ
M_i^*\big)^*\infconv q_i\Big)\circ L_i\Bigg)^*-q.
\end{equation}
\end{proposition}
\begin{proof}
The fact that $f\in\Gamma_0(\HH)$ follows from standard convex
analysis \cite{Livre1}. Now let $i\in I$. We derive from
\eqref{e:parasum}, \eqref{e:jjm41}, 
\cite[Cor.~16.30 and Thm.~16.47(i)]{Livre1}, and
Theorem~\ref{t:jjm1} that
\begin{align}
\label{e:jjm40}
\prox_{f_i}\infconv
\big(\partial g_i\infconv(M_i\pushfwd\partial h_i)\big)
&=\Big(\Idi+\partial f_i+\big(\partial g_i\infconv
(M_i\pushfwd\partial h_i)\big)^{-1}\Big)^{-1}\nonumber\\
&=\Big(\Idi+\partial f_i+(\partial g_i)^{-1}+
M_i\circ(\partial h_i)^{-1}\circ M_i^*\Big)^{-1}\nonumber\\
&=\Big(\Idi+\partial f_i+\partial g_i^*+
M_i\circ\partial h_i^*\circ M_i^*\Big)^{-1}\nonumber\\
&=\Big(\Idi+\partial\big(f_i+g_i^*+
h_i^*\circ M_i^*\big)\Big)^{-1}\nonumber\\
&=\prox_{f_i+g_i^*+h_i^*\circ M_i^*}\nonumber\\
&=\nabla\Big(\big(f_i+g_i^*+h_i^*\circ M_i^*\big)^*\infconv q_i
\Big).
\end{align}
Since $\prox_{f_i+g_i^*+h_i^*\circ M_i^*}\in\NL(\HH)$, 
\begin{equation}
\label{e:jjm42}
\nabla\Big(\Big(\big(f_i+g_i^*+h_i^*\circ 
M_i^*\big)^*\infconv q_i\Big)\circ L_i\Big)=
L_i^*\circ\prox_{f_i+g_i^*+h_i^*\circ M_i^*}\circ L_i
\end{equation}
has Lipschitz constant $\|L_i\|^2$. Altogether, 
\begin{equation}
\label{e:jjm7}
T=\sum_{i\in I}\omega_iL_i^*\circ\Big(\prox_{f_i}\infconv
\big(\partial g_i\infconv(M_i\pushfwd\partial h_i)\big)\Big)
\circ L_i 
=\nabla\Bigg(\sum_{i\in I}\omega_i\Big(\big(f_i+g_i^*+
h_i^*\circ M_i^*\big)^*\infconv q_i\Big)\circ L_i\Bigg)
\end{equation}
has Lipschitz constant $\sum_{i\in I}\omega_i\|L_i\|^2\leq 1$. 
In view of Theorem~\ref{t:jjm1}, the proof is complete.
\end{proof}

\begin{remark}
Let us highlight some special cases of 
Proposition~\ref{p:jjm2}. $\GG$ is a real Hilbert
space.
\label{r:jjm1}
\begin{enumerate}
\item
\label{r:jjm1i}
Let $(T_i)_{i\in I}$ be a finite family in $\PL(\HH)$
and let $(\omega_i)_{i\in I}$ be a finite family in $\rzeroun$
such that $\sum_{i\in I}\omega_i=1$. Then 
$\sum_{i\in I}\omega_iT_i\in\PL(\HH)$. This result is due to
Moreau \cite{Mor63a}. Connections with the proximal average are
discussed in \cite{Livre1}.
\item
\label{r:jjm1ix}
In \ref{r:jjm1i}, taking $I=\{1,2\}$, $T_1=T$, $T_2=\Id$, and
$\omega_1=\lambda\in\zeroun$ yields
$\Id+\lambda(T-\Id)\in\PL(\HH)$. The fact that the
under-relaxation of a proximity operator is a proximity 
operator 
appears in \cite{Smms05}. More precisely, it is shown there
that if $T=\prox_h$ for some $h\in\Gamma_0(\HH)$,
then $\Id+\lambda(T-\Id)=\prox_{f}$, where
$f=h\infconv(\lambda q)/(1-\lambda)$, which can now be seen 
as a consequence of \eqref{p:jjm3}.
\item
\label{r:jjm1viii}
Let $T_1$ and $T_2$ be in $\PL(\HH)$. Then
$(T_1-T_2+\Id)/2\in\PL(\HH)$. Indeed,
Proposition~\ref{p:jjm12}\ref{p:jjm12i} asserts that
$\Id-T_2\in\PL(\HH)$ and then \ref{r:jjm1i} that the average 
of $T_1$ and $\Id-T_2$ is also in $\PL(\HH)$.
\item
\label{r:jjm1vi}
Let $T\in\PL(\GG)$ and let $L\in\BL(\HH,\GG)$ be such that
$\|L\|\leq 1$. Then $L^*\circ T\circ L\in\PL(\HH)$.
\item
\label{r:jjm1vii}
Let $T\in\PL(\HH)$ and let $V$ be a closed vector subspace of
$\HH$. Then it follows from \ref{r:jjm1vi} that
$\proj_V\circ T\circ \proj_V\in\PL(\HH)$.
\item
\label{r:jjm1xii}
Suppose that $u\in\HH$ satisfies $0<\|u\|\leq 1$ and let
$R\in\PL(\RR)$. Set
$L\colon\HH\to\RR\colon x\mapsto\scal{x}{u}$ and
$T\colon\HH\to\HH\colon x\mapsto(R\scal{x}{u})u$. Then
\ref{r:jjm1vi} yields $T\in\PL(\HH)$.
\item
\label{r:jjm1v}
Let $M\in\BL(\GG,\HH)\smallsetminus\{0\}$,
let $f\in\Gamma_0(\HH)$, 
let $g\in\Gamma_0(\HH)$, and let 
$h\in\Gamma_0(\GG)$ be such that 
$0\in\sri(\dom h^*-M^*(\dom f\cap\dom g^*))$
and $0\in\sri(\dom f-\dom g^*)$.
Then $\prox_{f}\infconv(\partial g\infconv
(M\pushfwd\partial h))\in\PL(\HH)$. More specifically,
$\prox_{f}\infconv(\partial g\infconv
(M\pushfwd\partial h))=\prox_{f+g^*+h^*\circ M^*}$.
\item
\label{r:jjm1xi}
In \ref{r:jjm1v}, suppose that, in addition,
$g=\varphi^*\infconv q_{\HH}$ and $h=\psi^*\infconv q_{\GG}$, 
where $\varphi\in\Gamma_0(\HH)$ and $\psi\in\Gamma_0(\GG)$. Then
$\partial g=\{\prox_\varphi\}$, $\partial h=\{\prox_\psi\}$, and
we conclude that 
$\prox_{f}\infconv(\prox_\varphi\infconv
(M\pushfwd\prox_\psi))\in\PL(\HH)$. More specifically,
\begin{equation}
\prox_{f}\infconv\big(\prox_\varphi\infconv
(M\pushfwd\prox_\psi)\big)=\prox_{f+\varphi+(\psi+q_{\GG})
\circ M^*+q_{\HH}}
=\prox_{(f+\varphi+(\psi+q_{\GG})\circ M^*)/2}(\cdot/2)
\end{equation}
has Lipschitz constant $1/2$.
\item
In \ref{r:jjm1v}, suppose that, in addition, $\GG=\HH$,
$M=\Id$, $g=\varphi^*\infconv q$, and $h=\iota_{\{0\}}$, where
$\varphi\in\Gamma_0(\HH)$. Then
$\partial g=\{\prox_\varphi\}$, $\partial h=\{0\}^{-1}$, and we
conclude that 
$\prox_{f}\infconv\prox_\varphi\in\PL(\HH)$. More specifically,
\begin{equation}
\label{e:jjm44}
\prox_{f}\infconv\prox_\varphi
=\prox_{f+\varphi+q}
=\prox_{(f+\varphi)/2}(\cdot/2)
\end{equation}
has Lipschitz constant $1/2$. This result
appears in \cite[Cor.~25.35]{Livre1}.
\end{enumerate}
\end{remark}

Proposition~\ref{p:jjm2} allows us to interpret some algorithms
as simple instances of the standard proximal point algorithm
\eqref{e:martinet1} for convex minimization. 

\begin{example}
\label{ex:jjm83}
Let $K$ be a closed convex cone in $\HH$ with polar cone 
$K^\ominus$, let $V$ be a closed vector subspace of $\HH$,
and set
\begin{equation}
\label{e:f}
f=\bigg(\frac{1}{2}d_{K^\ominus}^2\circ \proj_V\bigg)^*-
\dfrac{\|\cdot\|^2}{2}
\end{equation}
and $T=\proj_V\circ \proj_K\circ \proj_V$.
Then it follows from Proposition~\ref{p:jjm2}
(see also Remark~\ref{r:jjm1}\ref{r:jjm1vii}) that $T=\prox_f$.
Now let $x_0\in V$ and consider the proximal point iterations
$(\forall n\in\NN)$ $x_{n+1}=\prox_f x_n$. Then the sequence
$(x_n)_{n\in\NN}$ is identical to that produced by the alternating
projection
algorithm $(\forall n\in\NN)$ $x_{n+1}=(\proj_V\circ\proj_K)x_n$. 
In \cite{Hund04}, a specific choice of $x_0$, $K$, and $V$ (the
latter being a closed hyperplane) lead to a sequence
$(x_n)_{n\in\NN}$ that was shown to converge weakly but not
strongly to the unique point in $K\cap V$, namely $0$. In turn,
\eqref{e:f} is a new example of a function for which the proximal
point algorithm converges weakly but not strongly. Alternative
constructions can be found in \cite{Baus04,Gule91}.
\end{example}

\begin{example}
\label{ex:jjm23}
Let $(C_i)_{i\in I}$ be a finite family of nonempty closed convex
subsets of $\HH$. The convex feasibility problem is to find a
point in $\bigcap_{i\in I}C_i$. When this problem has no
solution, a situation that arises frequently in signal recovery
due to inaccurate prior knowledge or measurement errors
\cite{Sign94}, one must find a surrogate minimization problem. 
Let us note that the standard method of periodic projections used
in consistent problems is of little value here as the limit 
cycles it generates do not minimize any function \cite{Jfan12}.
Let $x_0\in\HH$ and $\varepsilon\in\zeroun$. 
In \cite{Sign94}, it was proposed to minimize 
$(1/2)\sum_{i\in I}\omega_id_{C_i}^2$, where $(\omega_i)_{i\in I}$ 
are in $\rzeroun$ and satisfy $\sum_{i\in I}\omega_i=1$, via the
parallel projection method
\begin{equation}
\label{e:ppm}
(\forall n\in\NN)\quad
x_{n+1}=x_n+\lambda_n\Bigg(\sum_{i\in I}\omega_i
\proj_{C_i}x_n-x_n\Bigg),\quad\text{where}\quad
\varepsilon\leq\lambda_n\leq(2-\varepsilon).
\end{equation}
Now set $f=(\sum_{i\in I}\omega_i(\sigma_{C_i}\infconv q))^*-q$
and apply Proposition~\ref{p:jjm2} with $(\forall i\in I)$
$\GG_i=\KK_i=\HH$, $L_i=M_i=\Id$, $f_i=\iota_{C_i}$,
and $g_i=h_i=\iota_{\{0\}}$.
Then $\prox_f=\sum_{i\in I}\omega_i \proj_{C_i}$, and 
\eqref{e:ppm} therefore turns out to be just a relaxed instance 
of Martinet's proximal point algorithm \eqref{e:martinet1}. 
\end{example}

As noted in Example~\ref{ex:20}, a composition of proximity
operators is usually not a proximity operator. Likewise, the sum
of two proximity operators may not be in $\NL(\HH)$ and therefore
not in $\PL(\HH)$. The following propositions provide some
exceptions. We start with the identity
$\prox_{f_1}\circ\prox_{f_2}=\prox_{f_1+f_2}$, which is also
discussed in special cases in \cite{Siop07,Jsts07,Smms05,Yuyl13}.

\begin{proposition}
\label{p:kj1}
Let $T_1$ and $T_2$ be in $\PL(\HH)$, say $T_1=\prox_{f_1}$ and 
$T_2=\prox_{f_2}$ for some $f_1$ and $f_2$ in $\Gamma_0(\HH)$.
Suppose that $\dom f_1\cap\dom f_2\neq\emp$ and that
one of the following holds:
\begin{enumerate}
\item
\label{p:kj1-i}
$\HH=\RR$.
\item
\label{p:kj1i}
$(\forall x\in\dom\partial f_2)$ 
$\partial f_2(x)\subset\partial f_2(T_1 x)$.
\item
\label{p:kj1ii}
$(\forall (x,u)\in\gra\partial {f_1})$ 
$\partial f_2(x+u)\subset\partial f_2(x)$.
\item
\label{p:kj1iv}
$0\in\sri(\dom {f_1}-\dom f_2)$ and 
$(\forall (x,u)\in\gra\partial {f_1})$ 
$\partial f_2(x)\subset\partial f_2(x+u)$.
\end{enumerate}
Then $T_1\circ T_2\in\PL(\HH)$. More specifically, in cases
{\rm\ref{p:kj1i}--\ref{p:kj1iv}}, $T_1\circ T_2=\prox_{f_1+f_2}$.
\end{proposition}
\begin{proof}
\ref{p:kj1-i}: A function $T\colon\RR\to\RR$ belongs to 
$\PL(\RR)$ if and only if it is nonexpansive and increasing 
\cite{Siop07}. Since the composition of nonexpansive and
increasing functions is likewise, we obtain the claim. 

\ref{p:kj1i}--\ref{p:kj1iv}: \cite[Prop.~24.18]{Livre1}.
\end{proof}

\begin{proposition}
\label{p:maubert}
Let $C$ be a nonempty closed convex subset of $\HH$, let
$\phi\in\Gamma_0(\RR)$ be even, set 
$\varphi=\phi\circ\|\cdot\|+\sigma_C$, 
$T_1=\prox_{\phi\circ\|\cdot\|}$, and $T_2=\prox_{\sigma_C}$. 
Then $T_1\circ T_2\in\PL(\HH)$. More specifically,
$T_1\circ T_2=\prox_\varphi$. 
\end{proposition}
\begin{proof}
Let $x\in\HH$. If $\phi$ is constant, then $T_1=\Id$ and the result
is trivially true. We therefore assume otherwise, which allows us to 
derive from \cite[Prop.~2.2]{Nmtm09} that
\begin{equation}
\label{e:maubert1}
\prox_{\varphi}x=
\begin{cases}
\displaystyle{\frac{\prox_{\phi}d_C(x)}{d_C(x)}}
(x-\proj_Cx),&\text{if}\;\;d_C(x)>
\text{max\,\Argmin}\,\phi;\\
x-\proj_Cx,&\text{if}\;\;d_C(x)\leq{\text{max\,Argmin}}\,\phi.
\end{cases}
\end{equation}
For $C=\{0\}$, this yields 
\begin{equation}
\label{e:maubert2}
T_1x=
\begin{cases}
\displaystyle{\frac{\prox_{\phi}\|x\|}{\|x\|}}
x,&\text{if}\;\;\|x\|>
\text{max\,\Argmin}\,\phi;\\
x,&\text{if}\;\;\|x\|\leq{\text{max\,Argmin}}\,\phi.
\end{cases}
\end{equation}
Since Theorem~\ref{t:jjm0} yields $\Id-\proj_C=\prox_{\sigma_C}$, 
using \eqref{e:maubert1} and \eqref{e:maubert2}, we get
\begin{align}
\prox_{\varphi}x
&=
\begin{cases}
\displaystyle{\frac{\prox_{\phi}\|\prox_{\sigma_C}x\|}
{\|\prox_{\sigma_C}x\|}}
\prox_{\sigma_C}x,&\text{if}\;\;\|\prox_{\sigma_C}x\|>
\text{max\,\Argmin}\,\phi;\\
\prox_{\sigma_C}x,&\text{if}\;\;
\|\prox_{\sigma_C}x\|\leq{\text{max\,Argmin}}\,\phi
\end{cases}\nonumber\\
&=T_1(T_2x).
\end{align}
\end{proof}

\begin{remark}
Proposition~\ref{p:maubert} has important applications.
\begin{enumerate}
\item
It follows from \cite[Prop.~3.2(v)]{Siop07} that, if $\phi$ is
differentiable at $0$ with $\phi'(0)=0$, then $\prox_\varphi$,
which can be implemented explicitly via \eqref{e:maubert1}, is a
proximal thresholder on $C$: $(\forall x\in\HH)$ 
$\prox_\varphi x=0$ $\Leftrightarrow$ $x\in C$.
\item
Let $\gamma\in\RPP$, let $K$ be a nonempty closed convex 
cone in $\HH$, let $C$ be the polar cone of $K$, and let $x\in\HH$. 
Upon setting $\phi=\gamma|\cdot|$ in Proposition~\ref{p:maubert}
and using \cite[Examp.~24.20]{Livre1}, we obtain (see 
\cite[Lemma~2.2]{Jmaa18} for a different derivation)
\begin{equation}
\label{e:ecp}
\prox_{\gamma\|\cdot\|+\iota_K}x=
\big(\prox_{\gamma\|\cdot\|}\circ\proj_K\big)x=
\begin{cases}
\dfrac{\|\proj_Kx\|-\gamma}{\|\proj_Kx\|}\proj_Kx,
&\text{if}\;\;\|\proj_Kx\|>\gamma;\\
0,&\text{if}\;\;\|\proj_Kx\|\leq\gamma.
\end{cases}
\end{equation}
On the other hand, setting $\phi=\iota_{[-\gamma,\gamma]}$  in
Proposition~\ref{p:maubert} and using \cite[Examp.~3.18]{Livre1}, 
we obtain (see \cite[Sec.~7]{Buim17} for different derivations)
\begin{equation}
\label{e:ecp7}
\proj_{B(0;\gamma)\cap K}x=
\big(\proj_{B(0;\gamma)}\circ\proj_K\big)x=
\begin{cases}
\dfrac{\gamma}{\|\proj_Kx\|}\proj_Kx,
&\text{if}\;\;\|\proj_Kx\|>\gamma;\\
\proj_Kx,&\text{if}\;\;\|\proj_Kx\|\leq\gamma.
\end{cases}
\end{equation}
\end{enumerate}
\end{remark}

\begin{proposition}
\label{p:eduardo3}
Set $q=\|\cdot\|^2/2$, and let $T_1$ and $T_2$ be in $\PL(\HH)$,
say $T_1=\prox_{f_1}$ and 
$T_2=\prox_{f_2}$ for some $f_1$ and $f_2$ in $\Gamma_0(\HH)$.
Suppose that $0\in\sri(\dom f_1^*-\dom f_2^*)$ and that
\begin{equation}
\label{e:eduardo3}
(f_1^*+f_2^*)\infconv q=f_1^*\infconv q+f_2^*\infconv q.
\end{equation}
Then $T_1+T_2\in\PL(\HH)$. More specifically,
$T_1+T_2=\prox_{f_1\infconv f_2}$.
\end{proposition}
\begin{proof}
It follows from \cite[Prop.~15.7(i)]{Livre1} that
$f_1\infconv f_2\in\Gamma_0(\HH)$. In addition, we derive from 
Theorem~\ref{t:jjm0}, \eqref{e:eduardo3}, and 
\cite[Prop.~13.24(i)]{Livre1} that 
$T_1+T_2
=\nabla\big(f_1^*\infconv q\big)+\nabla\big(f_2^*\infconv q\big)
=\nabla\big(f_1^*\infconv q+f_2^*\infconv q\big)
=\nabla\big((f_1^*+f_2^*)\infconv q\big)
=\nabla\big((f_1\infconv f_2)^*\infconv q\big)
=\prox_{f_1\infconv f_2}$,
as claimed.
\end{proof}

\begin{remark}
\label{r:18}
Let $C_1$ and $C_2$ be nonempty closed convex subsets of $\HH$,
and set $f_1=\iota_{C_1}$ and $f_2=\iota_{C_2}$. Then the 
conclusion of Proposition~\ref{p:eduardo3} is that 
$\proj_{C_1}+\proj_{C_2}=\proj_{C_1+C_2}$. This property is
discussed in \cite{Zara74} (for cones), in 
\cite[Prop.~29.6]{Livre1}, and in the recently posted 
paper \cite{Buim18}.
\end{remark}

\subsection{Self-dual classes of firmly nonexpansive operators}

Let us call a subclass $\TL(\HH)$ of $\JL(\HH)$ 
\emph{self-dual} if
$(\forall T\in\TL(\HH))$ $\Id-T\in\TL(\HH)$. This property plays
an important role in our paper. 

It is clear from \eqref{e:firm} that $\JL(\HH)$ is self-dual. This
can also be recovered from Theorem~\ref{t:minty2} and \eqref{e:ja}.
As seen in Proposition~\ref{p:jjm12}\ref{p:jjm12i}, $\PL(\HH)$ 
is also self-dual. Now let $T\in\KL(\HH)$. Then there exists a 
nonempty closed convex cone $K\subset\HH$ such that 
$T=\proj_K$ and Moreau's conical decomposition expresses the
projector onto the polar cone $K^{\ominus}$ as 
$\proj_{K^\ominus}=\Id-\proj_K$ \cite{Mor62a}. This shows that
$\KL(\HH)$ is self-dual. 
Likewise, it follows from the standard Beppo Levi
orthogonal decomposition of $\HH$ \cite{Levi06} that the class
$\VL(\HH)$ of projectors onto closed vector subspaces of $\HH$ is
self-dual. We thus obtain the nested self-dual classes
\begin{equation}
\label{e:self-dual}
\VL(\HH)\subset\KL(\HH)\subset\PL(\HH)\subset\JL(\HH).
\end{equation}
Self-duality properties were investigated in \cite{Baus12}, where 
other classes were identified and studied in depth. 
In particular, let $A\colon\HH\to 2^{\HH}$ be maximally
monotone. Then $A$ is paramonotone if and only if $A^{-1}$ is
\cite[Prop.~22.2(i)]{Livre1}. As a result, it follows from 
\eqref{e:ja} that the class 
$\JL_{\text{para}}(\HH)$ of resolvents of paramonotone maximally 
monotone operators from $\HH$ to $2^{\HH}$ is self-dual
\cite{Baus12} and since subdifferentials are paramonotone,
we have $\PL(\HH)\subset\JL_{\text{para}}(\HH)\subset\JL(\HH)$.
Likewise, since $A$ is $3^*$~monotone if and only if $A^{-1}$ is
\cite[Prop.~25.19(i)]{Livre1}, and since subdifferentials 
are $3^*$~monotone, the class $\JL_{3^*}(\HH)$ of resolvents of
$3^*$~monotone maximally monotone operators from $\HH$ to $2^{\HH}$
is self-dual and satisfies
$\PL(\HH)\subset\JL_{3^*}(\HH)\subset\JL(\HH)$.

Although our primary objective in Section~\ref{sec:33} was to 
investigate transformations on the class $\PL(\HH)$, similar 
questions could be asked about other self-dual classes. In this 
spirit, Zarantonello \cite{Zara74} has studied some transformations
in $\KL(\HH)$. Let $T_1$ and $T_2$ be in $\KL(\HH)$, say
$T_1=\proj_{K_1}$ and $T_2=\proj_{K_2}$.
In connection with Proposition~\ref{p:kj1}, he has shown that 
$\{T_1\circ T_2,T_2\circ T_1\}\subset\KL(\HH)$ if and only if 
$T_2\circ T_1=T_1\circ T_2$, in which case 
$T_1\circ T_2=\proj_{K_1\cap K_2}$ \cite{Zara73}. On the other
hand, in this context, the conclusion of
Proposition~\ref{p:eduardo3}, which states that
$T_1+T_2=\proj_{K_1+K_2}\in\KL(\HH)$, is discussed
in \cite{Zara74}.

The proximity-preserving transformations studied in 
Section~\ref{sec:33} have natural resolvent-preserving
counterparts.  For instance, mimicking the pattern of
Remark~\ref{r:jjm1}\ref{r:jjm1v} and using
\cite[Thm.~25.3]{Livre1}, one shows that, if 
$T=J_{A}\in\JL(\HH)$, 
$B\in\ML(\HH)$, $M\in\BL(\GG,\HH)\smallsetminus\{0\}$, and 
$C\in\ML(\GG)$, then $T\infconv(B\infconv(M\pushfwd C))
=J_{A+B^{-1}+M\circ C^{-1}\circ M^*}\in\JL(\HH)$
provided that the cones generated by 
$\dom C^{-1}-M^*(\dom A\cap\dom B^{-1})$ and by
$\dom A-\dom B^{-1}$ are closed vector subspaces.

\section{Monotone operators in convex optimization}
\label{sec:4}

In this section we present several examples of maximally monotone
operators which are not subdifferentials and which play 
fundamental and indispensable roles in the analysis and 
the numerical solution of convex optimization problems.  We preface
these examples with a brief overview of classical splitting methods
\cite{Livre1} which depend less critically on monotone operator
theory.

\subsection{The interplay between splitting methods for convex
optimization and monotone inclusion problems}
\label{sec:split}

The proximal point algorithm \eqref{e:martinet1} was first
developed for convex optimization. It was extended 
in \cite{Roc76a} to solve the inclusion problem 
\eqref{e:zero} for an operator $A\in\ML(\HH)$ such that 
$\zer A\neq\emp$ via the iteration
\begin{equation}
\label{e:ppa}
x_0\in\HH\quad\text{and}\quad(\forall n\in\NN)\quad
x_{n+1}=J_{\gamma_n A}x_n, \quad\text{where}\quad
\gamma_n\in\RPP.
\end{equation}
However, the algorithmic theory for the case of monotone
inclusions does not subsume that for the case of convex
optimization. Thus, as shown in \cite{Brez78}, the weak
convergence of $(x_n)_{n\in\NN}$ to a point in $A$ holds when
$\sum_{n\in\NN}\gamma_n^2=\pinf$, and this condition can be
weakened to $\sum_{n\in\NN}\gamma_n=\pinf$ if $A\in\SL(\HH)$
(see also \cite{Gule91} for finer properties in the 
subdifferential case). This is explained by the fact that, 
given $z\in\zer A$, \eqref{e:ppa} and Theorem~\ref{t:minty2}
yield
\begin{equation}
(\forall n\in\NN)\quad\|x_{n+1}-z\|^2\leq\|x_n-z\|^2-
\|J_{\gamma_n A}x_n-x_n\|^2
\end{equation}
for a general $A\in\ML(\HH)$ while, when $A=\partial f$ for some
$f\in\Gamma_0(\HH)$, it can be sharpened to 
\begin{equation}
(\forall n\in\NN)\quad\|x_{n+1}-z\|^2\leq\|x_n-z\|^2-
\|J_{\gamma_n A}x_n-x_n\|^2-2\gamma_n
\big(f(x_{n+1})-\text{inf}\: f(\HH)\big).
\end{equation}
Going back to the discussion of Section~\ref{sec:3}, this sheds a
different light on the differences between $\SL(\HH)$ and
$\ML(\HH)$. Naturally, the applicability of \eqref{e:ppa}
depends on the ease of implementation of the resolvents
$(J_{A_n})_{n\in\NN}$. A more structured inclusion problem is
the following.

\begin{problem}
\label{prob:1}
Let $A\in\ML(\HH)$ and $B\in\ML(\HH)$ be such that 
$0\in\ran(A+B)$. Find a zero of $A+B$. 
\end{problem}

Generally speaking, when replacing monotone operators by
subdifferentials in certain inclusion problems, one
recovers a convex minimization problem provided some constraint
qualification holds \cite{Livre1}. In this regard, we shall also
consider the following convex optimization problem.

\begin{problem}
\label{prob:8}
Let $f$ and $g$ be functions in $\Gamma_0(\HH)$ such that
$0\in\ran(\partial f+\partial g)$. Find a
minimizer of $f+g$ over $\HH$.
\end{problem}

There are three classical methods for solving Problem~\ref{prob:1}, 
which we present here in simple forms (see
\cite{Siop11,Cham15,Jmaa15} and the references therein for
refinements). All three methods produce a sequence
$(x_n)_{n\in\NN}$ which converges weakly to a zero of $A+B$
\cite{Livre1,Tsen91,Tsen00}, but they involve different
assumptions on $B$. Let us stress that the importance of these
three splitting methods 
is not only historical: many seemingly different splitting 
methods are just, explicitly or implicitly, 
reformulations of these basic schemes in alternate settings
(e.g., product spaces, dual spaces, primal-dual spaces, 
renormed spaces, or a combination thereof); see 
\cite{Sico10,Atto18,Siop11,Joca09,Icip14,Svva10,Svva12,Opti14,%
Cond13,Ecks92,Gaba83,Leno17,Roc76b,Tsen00,Bang13} and the 
references therein for specific examples.

\begin{itemize}
\item
{\bfseries Forward-backward splitting.}
In Problem~\ref{prob:1}, suppose that $B\colon\HH\to\HH$ and that
$\beta^{-1} B\in\FL(\HH)$ for some $\beta\in\RPP$.
Let $x_0\in\HH$ and $\varepsilon\in\left]0,2/(\beta+1)\right[$,
and iterate
\begin{equation}
\label{e:fb1}
\begin{array}{l}
\text{for}\;n=0,1,\ldots\\
\left\lfloor
\begin{array}{l}
\varepsilon\leq\gamma_n\leq(2-\varepsilon)/\beta\\
y_n=x_n-\gamma_n Bx_n\\
x_{n+1}=J_{\gamma_n A}y_n.
\end{array}
\right.\\[2mm]
\end{array}
\end{equation}
In view of Theorem~\ref{Trio9-}, the assumptions on $B$ translate
into the fact that, in Problem~\ref{prob:8}, $f_2\colon\HH\to\RR$
is differentiable and that $B=\nabla f_2$ is
$\beta$-Lipschitzian, while $A=\partial f_1$.
\item
{\bfseries Tseng's forward-backward-forward splitting.}
In Problem~\ref{prob:1}, suppose that $B\colon\HH\to\HH$ 
is $\beta$-Lipschitzian for some $\beta\in\RPP$.
Let $x_0\in\HH$ and $\varepsilon\in\left]0,1/(\beta+1)\right[$,
and iterate
\begin{equation}
\label{e:fbf}
\begin{array}{l}
\text{for}\;n=0,1,\ldots\\
\left\lfloor
\begin{array}{l}
\varepsilon\leq\gamma_n\leq(1-\varepsilon)/\beta\\
y_n={x}_n-\gamma_n Bx_n\\
p_n=J_{\gamma_n A}{y}_n\\
q_n=p_n-\gamma_nBp_n\\
x_{n+1}=x_n-y_n+q_n.
\end{array}
\right.\\[2mm]
\end{array}
\end{equation}
\item
{\bfseries Douglas-Rachford splitting.}
Let $x_0\in\HH$ and $\gamma\in\RPP$, and iterate
\begin{equation}
\label{e:dr}
\begin{array}{l}
\text{for}\;n=0,1,\ldots\\
\left\lfloor
\begin{array}{l}
x_n=J_{\gamma B}y_n\\
z_n=J_{\gamma A}(2x_n-y_n)\\
y_{n+1}=y_n+z_n-x_n.
\end{array}
\right.\\[2mm]
\end{array}
\end{equation}
\end{itemize}

Historically, the forward-backward method grew out of the
projected gradient method in convex optimization \cite{Levi66},
and the first version for Problem~\ref{prob:1} was proposed in
\cite{Merc79}. Another example of a monotone operator splitting 
method that evolved from convex optimization is Dykstra's method
\cite{Paci08}, which was first devised for indicator functions
in \cite{Boyl86}. By contrast, the forward-backward-forward
\cite{Tsen00} and
Douglas-Rachford \cite{Lion79} methods were developed directly for
Problem~\ref{prob:1}, and then specialized to
Problem~\ref{prob:8}. In principle, however, even though 
monotone inclusions provide a more synthetic and natural 
framework, it is possible (at least {\em a posteriori}) 
to derive their convergence in the scenario of
Problem~\ref{prob:8} from optimization concepts only, without
invoking monotone operator theory. Nonetheless, 
non-subdifferential maximally monotone operators may still be 
at play. For instance, note that in the Douglas-Rachford 
algorithm \eqref{e:dr}, we have
\begin{equation}
\label{e:dr2}
(\forall n\in\NN)\quad y_{n+1}=Ty_n,\quad\text{where}\quad
T=\frac{(2J_{\gamma A}-\Id)\circ(2J_{\gamma B}-\Id)+\Id}{2}.
\end{equation}
Upon invoking \eqref{e:dr7} and Theorem~\ref{t:minty2}, we see 
that $T\in\JL(\HH)$ and that there exists 
$C\in\ML(\HH)$ such that $T=J_C$, namely 
$C=T^{-1}-\Id$. Hence
\begin{equation}
\label{e:dr4}
(\forall n\in\NN)\quad y_{n+1}=J_Cy_n,\quad\text{where}\quad
C=T^{-1}-\Id. 
\end{equation}
In other words, $(y_n)_{n\in\NN}$ is produced by an instance of
the proximal point algorithm \eqref{e:ppa} and, in this sense,
the dynamics of the Douglas-Rachford algorithm are implicitly 
governed by a maximally monotone operator (as seen in
Example~\ref{ex:20}, this operator is typically not in $\SL(\HH)$,
even if $A$ and $B$ are). This
observation, which was made in \cite{Ecks92}, has actually a much
more general scope. Indeed, as shown in \cite{Opti04}, several
operator splitting algorithms are driven by successive
approximations of an averaged operator $T\colon\HH\to\HH$, i.e., an
operator 
of the form $T=(1-\alpha)\Id+\alpha R$, where $R\colon\HH\to\HH$ 
is nonexpansive and $\alpha\in\zeroun$ (further examples are 
found in more recent papers such as \cite{Davi17} and 
\cite{Ragu13}). We derive from 
Theorem~\ref{t:minty2} that there exists 
$C\in\ML(\HH)$ (and $C\notin\SL(\HH)$ in general) such 
that $R=2J_C-\Id$, namely $C=((R+\Id)/2)^{-1}-\Id$. Therefore,
$T=\Id+2\alpha(J_C-\Id)$. In turn, a sequence 
$(x_n)_{n\in\NN}$ produced by the successive 
approximations of $T$ is generated implicitly by the
relaxed resolvent iteration
\begin{equation}
\label{e:G1973}
(\forall n\in\NN)\quad
x_{n+1}=x_n+\lambda(J_Cx_n-x_n),\quad\text{where}\quad
\lambda=2\alpha\quad\text{and}\quad
C=\bigg(\Id+\frac{1}{2\alpha}(T-\Id)\bigg)^{-1}-\Id.
\end{equation}
For example, let us consider the forward-backward algorithm
\eqref{e:fb1} with a fixed proximal parameter 
$\gamma\in\left]0,2/\beta\right[$. Then
\begin{equation}
\label{e:fb3}
(\forall n\in\NN)\quad
x_{n+1}=Tx_n\quad\text{where}\quad 
T=J_{\gamma A}\circ(\Id-\gamma B). 
\end{equation}
Furthermore, $T$ is averaged with constant
$\alpha=2/(4-\beta\gamma)$ \cite{Jmaa15}. Altogether, the
forward-backward iteration \eqref{e:fb3} is an instance of the
relaxed proximal point algorithm
\begin{multline}
\label{e:fb4}
(\forall n\in\NN)\quad
x_{n+1}=x_n+\lambda(J_Cx_n-x_n),\quad\text{where}\quad
\lambda=\frac{4}{4-\beta\gamma}\\
\text{and}\quad
C=\bigg(\frac{\beta\gamma}{4}\Id+\frac{4-\beta\gamma}{4}
\big(J_{\gamma A}\circ(\Id-\gamma B)\big)\bigg)^{-1}-\Id.
\end{multline}

\subsection{Rockafellar's saddle function operator}

The following result is due to Rockafellar \cite{Rock70,Roc71d}
(he actually used a somewhat more general notion of 
closedness, made precise in these papers, for the function 
$\mathcal{L}$).

\begin{theorem}[Rockafellar]
\label{t:rocky70}
Let $\HH_1$ and $\HH_2$ be real Hilbert spaces, let 
$\mathcal{L}\colon\HH_1\oplus\HH_2\to\RXX$ be such that, 
for every $x_1\in\HH_1$
and every $x_2\in\HH_2$, 
$-\mathcal{L}(x_1,\cdot)\in\Gamma_0(\HH_2)$ and
$\mathcal{L}(\cdot,x_2)\in\Gamma_0(\HH_1)$. Set
\begin{equation}
\label{e:rocky70}
(\forall x_1\in\HH_1)(\forall x_2\in\HH_2)\quad 
\boldsymbol{A}(x_1,x_2)=
\partial\mathcal{L}(\cdot,x_2)(x_1)\times
\partial(-\mathcal{L}(x_1,\cdot))(x_2).
\end{equation}
Then $\boldsymbol{A}\in\ML(\HH_1\oplus\HH_2)$ and
\begin{equation}
\zer\boldsymbol{A}=\menge{(x_1,x_2)\in\HH_1\oplus\HH_2}
{\mathcal{L}(x_1,x_2)=\inf \mathcal{L}(\HH_1,x_2)=\sup
\mathcal{L}(x_1,\HH_2)}
\end{equation}
is the set of saddle points of $\mathcal{L}$.
\end{theorem}

A geometrical interpretation of \eqref{e:rocky70} is that
$(u_1,u_2)\in \boldsymbol{A}(x_1,x_2)$ if and only if $(x_1,x_2)$
is a saddle point of the convex-concave function
$(x'_1,x'_2)\mapsto\mathcal{L}(x'_1,x'_2)-
\scal{x'_1}{u_1}+\scal{x'_2}{u_2}$. The maximally monotone
operator $\boldsymbol{A}$ of \eqref{e:rocky70} is deeply rooted
in convex optimization due to the foundational role it plays in
Lagrangian theory and duality schemes
\cite{Livre1,Rock70,Rock74,Roc76b}. Yet, as the following example
shows, it is not a subdifferential. 

\begin{example}
\label{ex:1}
In Theorem~\ref{t:rocky70}, set $\HH_1=\HH_2=\RR$ and 
$\mathcal{L}\colon(x_1,x_2)\mapsto x_1^2-x_1x_2$. 
Then \eqref{e:rocky70} yields
$\boldsymbol{A}\colon(x_1,x_2)\mapsto(2x_1-x_2,x_1)$. Thus,
$\boldsymbol{A}\in\BL(\RR^2)$ is positive and not 
self-adjoint. It follows from Lemma~\ref{l:1} that 
$\boldsymbol{A}\in\ML(\HH)\smallsetminus\SL(\HH)$.
\end{example}

The idea of using the proximal point algorithm \eqref{e:ppa} with
the operator $\boldsymbol{A}$ of \eqref{e:rocky70} to find a
saddle point of $\mathcal{L}$ was proposed by Rockafellar in
\cite{Roc76a}. In \cite{Roc76b}, he applied it to the concrete
problem of minimizing a convex function subject to convex
inequality constraints, using the ordinary Lagrangian as a saddle
function. The resulting algorithm is known as the proximal method
of multipliers.

\subsection{Spingarn's partial inverse operator}

Let $A\in\ML(\HH)$, let $V$ be a closed vector subspace of $\HH$, 
and let the partial inverse of $A$ with respect to $V$ be the 
operator $A_V\in\ML(\HH)$ defined in
\eqref{e:fiwejfweo}. As discussed in \cite{Spin83}, problems of 
the form
\begin{equation}
\label{e:cauq1}
\text{find}\;\;x\in V\;\;\text{and}\;\;
u\in V^\bot\;\;\text{such that}\;\;u\in Ax
\end{equation}
can be solved by applying the proximal point algorithm
\eqref{e:ppa} to  $A_V$; this method is known
as the method of partial inverses, and it has strong connections
with the Douglas-Rachford algorithm \cite{Ecks92,Lawr87,Mahe95}. 
For instance, if $A=\partial f$ for some $f\in\Gamma_0(\HH)$ such
that $f$ admits a minimizer over $V$ and 
$0\in\sri(V-\dom f)$, \eqref{e:cauq1} reduces to finding
a solution of the Fenchel dual pair
\begin{equation}
\label{e:cauq3}
\minimize{x\in V}{f(x)}\quad\text{and}\quad
\minimize{u\in V^\bot}{f^*(u)}.
\end{equation}
In this case, given $x_0\in V$ and $u_0\in V^\bot$, 
the method of partial inverses iterates
\begin{equation}
\label{e:cauq2}
\begin{array}{l}
\text{for}\;n=0,1,\ldots\\
\left\lfloor
\begin{array}{l}
y_n=\prox_f(x_n+u_n)\\
v_n=x_n+u_n-y_n\\
(x_{n+1},u_{n+1})=(\proj_Vy_n,\proj_{V^\bot}v_n),
\end{array}
\right.\\[2mm]
\end{array}
\end{equation}
and the sequences $(x_n)_{n\in\NN}$ and $(u_n)_{n\in\NN}$ 
converge weakly to a solution to the primal and dual problems
in \eqref{e:cauq3} \cite[Prop.~28.2]{Livre1}. 
This algorithm has many applications in convex optimization, e.g., 
\cite{Idri89,Lema89,Leno17,Spin83,Spin85,Spin87}. 
It also constitutes the basic
building block of the progressive hedging algorithm in stochastic
programming \cite{Rock91}. Thus, despite its apparent simplicity, 
this partial inverse approach is quite powerful and it can 
tackle the following primal-dual problem.

\begin{problem}
\label{prob:4}
Let $I$ be a nonempty finite set, and let $(\HH_i)_{i\in I}$ and 
$\GG$ be real Hilbert spaces. Let $r\in\GG$, let 
$g\in\Gamma_0(\GG)$, and, for every $i\in I$,
let $z_i\in\HH_i$, let $f_i\in\Gamma_0(\HH_i)$, 
and let $L_{i}\in\BL(\HH_i,\GG)$. 
Solve the primal problem
\begin{equation}
\minimize{(\forall i\in I)\;x_i\in\HH_i}
{\sum_{i\in I}\big(f_i(x_i)-\scal{x_i}{z_i}\big)
+g\Bigg(\sum_{i\in I}L_ix_i-r\Bigg)},
\end{equation}
together with the dual problem
\begin{equation}
\minimize{v\in\GG}{\sum_{i\in I}f_i^*\big(z_i-L_i^*v\big)
+g^*(v)+\scal{v}{r}}.
\end{equation}
\end{problem}

It is shown in \cite{Optl14} that, when applied to a version of 
\eqref{e:cauq1} suitably reformulated in a product space,
\eqref{e:cauq2} yields a proximal splitting algorithm that solves
Problem~\ref{prob:4} and employs the operators 
$(\prox_{f_i})_{i\in I}$, $\prox_g$, 
$(L_i)_{i\in I}$, and $(L^*_i)_{i\in I}$ 
separately.

The backbone of all the above-mentioned applications of the
method of partial inverses to convex optimization is the partial
inverse of an operator in $\SL(\HH)$. As seen in
Example~\ref{ex:21}, this maximally monotone operator is not in
$\SL(\HH)$ in general.

\subsection{Primal-dual algorithm for mixed composite minimization}
We re-examine through the lens of the maximally monotone 
saddle function operator \eqref{e:rocky70} a mixed
composite minimization problem proposed and studied in
\cite{Svva12} with different tools. 

\begin{problem}
\label{prob:2}
Let $f\in\Gamma_0(\HH)$, let $h\colon\HH\to\RR$ be convex and
differentiable with a $\mu$-Lipschitzian gradient for some
$\mu\in\RPP$, let $\GG$ be a real Hilbert space, let 
$g\in\Gamma_0(\GG)$, and let $\ell\in\Gamma_0(\GG)$ 
be $1/\nu$-strongly convex for some $\nu\in\RPP$. Suppose that 
$0\neq L\in\BL(\HH,\GG)$ and that 
\begin{equation}
\label{e:cq1}
0\in\ran\big(\partial f+L^*\circ(\partial g\infconv
\partial\ell)\circ L+\nabla h\big).
\end{equation}
Consider the problem
\begin{equation}
\label{e:primalvar}
\minimize{x\in\HH}{f(x)+(g\infconv\ell)(Lx)+h(x)}, 
\end{equation}
and the dual problem
\begin{equation}
\label{e:dualvar}
\minimize{v\in\GG}{\big(f^*\infconv h^*\big)
(-L^*v)+g^*(v)+\ell^*(v)}.
\end{equation}
\end{problem}

From a numerical standpoint, solving \eqref{e:primalvar} is
challenging as it involves five objects (four functions, three of 
which are nonsmooth, and a linear operator), while traditional 
proximal splitting techniques are limited to two objects; see
\eqref{e:fb1}--\eqref{e:dr}.
In \cite{Svva12}, Problem~\ref{prob:2} was analyzed and solved as 
an instance of a more general primal-dual inclusion problem 
involving monotone operators, which was reformulated as that of 
finding a zero of the sum of two operators in
$\ML(\HH\oplus\GG)$. Let us stress that, even in the special case 
of Problem~\ref{prob:2}, this inclusion problem still involves
operators which are not subdifferentials. To see this, we now
propose an alternative derivation of the results of
\cite[Sect.~4]{Svva12} using the saddle function formalism of
Theorem~\ref{t:rocky70}. Following the same pattern 
as in \cite[Examp.~11]{Rock74} (with the conventions of 
\cite[Prop.~19.20]{Livre1}), we define the Lagrangian of 
Problem~\ref{prob:2} as
\begin{align}
\label{e:Lagrangian11}
\mathcal{L}\colon\HH\oplus\GG&\to\RXX\nonumber\\
(x,v)
&\mapsto 
\begin{cases}
\minf, &\text{if}\;\;x\in\dom f\;\text{and}\;
v\notin\dom g^*\cap\dom\ell^*;\\
f(x)+h(x)+\scal{Lx}{v}-g^*(v)-\ell^*(v), 
&\text{if}\;\;x\in\dom f\;\text{and}\;v\in\dom g^*\cap\dom\ell^*;\\
\pinf, &\text{if}\;\;x\notin\dom f,
\end{cases}
\end{align}
and observe that it satisfies the assumptions of
Theorem~\ref{t:rocky70}. In turn, using standard subdifferential
calculus \cite{Livre1}, we deduce that the associated maximally
monotone operator $\boldsymbol{A}$ of \eqref{e:rocky70} is
\begin{align}
\label{e:svva12a}
\boldsymbol{A}\colon\HH\oplus\GG\to 2^{\HH\oplus\GG}\colon(x,v)
&\mapsto\partial\mathcal{L}(\cdot,v)(x)\times
\partial(-\mathcal{L}(x,\cdot))(v)\nonumber\\
&=\big(\partial f(x)+\nabla h(x)+L^*v,
\partial g^*(v)+\nabla\ell^*(v)-Lx\big).
\end{align}
It is noteworthy that this operator admits a Borwein-Wiersma 
decomposition \eqref{e:B-W}, namely
\begin{equation}
\label{e:B-W2}
\boldsymbol{A}=\partial\boldsymbol{\varphi}+\boldsymbol{S},
\quad\text{where}\quad
\begin{cases}
\boldsymbol{\varphi}\colon\HH\oplus\GG\to\RX\colon(x,v)\mapsto 
f(x)+h(x)+g^*(v)+\ell^*(v)\\
\boldsymbol{S}\colon\HH\oplus\GG\to\HH\oplus\GG
\colon(x,v)\mapsto(L^*v,-Lx).
\end{cases}
\end{equation}
Here $\boldsymbol{\varphi}\in\Gamma_0(\HH\oplus\GG)$ and 
$\boldsymbol{S}\in\BL(\HH\oplus\GG)$
is nonzero and skew, which shows that 
$\boldsymbol{A}\notin\SL(\HH\oplus\GG)$ 
by virtue of Lemma~\ref{l:1}.  
By Theorem~\ref{t:rocky70}, a zero $(x,v)$ of
$\boldsymbol{A}$ is a saddle point of $\mathcal{L}$, which
implies that $(x,v)$ solves Problem~\ref{prob:2}, i.e., $x$
solves \eqref{e:primalvar} and $v$ solves \eqref{e:dualvar}.
However, the decomposition \eqref{e:B-W2} does not lend itself
easily to splitting 
methods as they would require computing
$J_{\partial\boldsymbol{\varphi}}=\prox_{\boldsymbol{\varphi}}=
\prox_{f+h}\times\prox_{g^*+\ell^*}$, which does not admit a closed
form expression in general. A more judicious decomposition of
$\boldsymbol{A}$ is
\begin{equation}
\label{e:svva12f}
\boldsymbol{A}=\partial \boldsymbol{f}+\boldsymbol{B},
\quad\text{where}\quad
\begin{cases}
\boldsymbol{f}\colon\HH\oplus\GG\to\RX\colon(x,v)\mapsto 
f(x)+g^*(v)\\
\boldsymbol{B}\colon\HH\oplus\GG\to\HH\oplus\GG\colon(x,v)\mapsto
(\nabla h(x)+L^*v,\nabla\ell^*(v)-Lx).
\end{cases}
\end{equation}
Note that $\boldsymbol{f}\in\Gamma_0(\HH\oplus\GG)$ and that
computing $J_{\partial\boldsymbol{f}}=\prox_{\boldsymbol{f}}
=\prox_{f}\times\prox_{g^*}$ requires
only the ability to compute $\prox_{f}$ and 
$\prox_{g^*}=\Id-\prox_g$. Furthermore \cite{Svva12}, 
\begin{equation}
\label{e:lipB}
\boldsymbol{B}\in\ML(\HH\oplus\GG)\;
\text{is monotone and $\beta$-Lipschitzian with}\;
\beta=\text{max}\{\mu,\nu\}+\|L\|. 
\end{equation}
This structure makes the task of finding a zero of
$\boldsymbol{A}$ amenable to the
forward-backward-forward algorithm \eqref{e:fbf},
which requires one evaluation of $\prox_{\gamma \boldsymbol{f}}$
and two evaluations of $\boldsymbol{B}$ at each iteration. 
As seen in Section~\ref{sec:split}, given
$\varepsilon\in\left]0,1/(\beta+1)\right[$,
the forward-backward-forward algorithm constructs a sequence 
$(\boldsymbol{x}_n)_{n\in\NN}$ which converges weakly to 
a point in $\zer\boldsymbol{A}$ via the recursion
\begin{equation}
\label{e:FBF}
\begin{array}{l}
\text{for}\;n=0,1,\ldots\\
\left\lfloor
\begin{array}{l}
\varepsilon\leq\gamma_n\leq(1-\varepsilon)/\beta\\
\boldsymbol{y}_n=\boldsymbol{x}_n-\gamma_n 
\boldsymbol{B}\boldsymbol{x}_n\\
\boldsymbol{p}_n=\prox_{\gamma_n \boldsymbol{f}}\boldsymbol{y}_n\\
\boldsymbol{q}_n=\boldsymbol{p}_n-\gamma_n 
\boldsymbol{B}\boldsymbol{p}_n\\
\boldsymbol{x}_{n+1}=\boldsymbol{x}_n-\boldsymbol{y}_n
+\boldsymbol{q}_n.
\end{array}
\right.\\[2mm]
\end{array}
\end{equation}
Now set $(\forall n\in\NN)$
$\boldsymbol{x}_n=(x_n,v_n)$,
$\boldsymbol{y}_n=(y_{1,n},y_{2,n})$,
$\boldsymbol{p}_n=(p_{1,n},p_{2,n})$, and
$\boldsymbol{q}_n=(q_{1,n},q_{2,n})$.
Then, in view of \eqref{e:svva12f}, \eqref{e:FBF} assumes the
form of the primal-dual method of \cite[Sect.~4]{Svva12}, namely
\begin{equation}
\label{e:blackpagepart3}
\begin{array}{l}
\text{for}\;n=0,1,\ldots\\
\left\lfloor
\begin{array}{l}
\varepsilon\leq\gamma_n\leq(1-\varepsilon)/\beta\\
y_{1,n}=x_n-\gamma_n\big(\nabla h(x_n)+L^*v_{n}\big)\\
y_{2,n}=v_{n}+\gamma_n\big(Lx_n-\nabla\ell^{*}(v_{n})\big)\\
p_{1,n}=\prox_{\gamma_n f}y_{1,n}\\
p_{2,n}=\prox_{\gamma_n g^{*}}y_{2,n}\\
q_{1,n}=p_{1,n}-\gamma_n\big(\nabla h(p_{1,n})+
L^*p_{2,n}\big)\\
q_{2,n}=p_{2,n}+\gamma_n\big(Lp_{1,n}-\nabla\ell^{*}
(p_{2,n})\big)\\
x_{n+1}=x_n-y_{1,n}+q_{1,n}\\
v_{n+1}=v_{n}-y_{2,n}+q_{2,n}.
\end{array}
\right.\\
\end{array}
\end{equation}
We conclude that $(x_n)_{n\in\NN}$ converges weakly to a solution
$x$ to \eqref{e:primalvar} and that $(v_n)_{n\in\NN}$ converges 
weakly to a solution $v$ to \eqref{e:dualvar}. 

\begin{remark}
\label{r:1}
Let us make a few observations regarding Problem~\ref{prob:2}
and the iterative method \eqref{e:blackpagepart3}.
\begin{enumerate}
\item
\label{r:1i}
Algorithm~\eqref{e:blackpagepart3} achieves full splitting of the
functions and of the linear operators. In addition, all the smooth
functions are activated via explicit gradient steps, while the
nonsmooth ones are activated via their proximity operator.
\item
\label{r:1ii}
The special case when $\ell=\iota_{\{0\}}$ and $h=0$ leads to the
monotone+skew decomposition approach of \cite{Siop11}. As
discussed in \cite[Rem.~2.9]{Siop11}, in this case the use
Douglas-Rachford algorithm \eqref{e:dr} can also be contemplated 
since the resolvents $J_{\gamma\partial\boldsymbol{f}}$ and 
$J_{\gamma\boldsymbol{S}}$ of the operators in \eqref{e:B-W2} have
explicit forms.
\item
\label{r:1iii}
In \cite{Bang13}, Problem~\ref{prob:2} is also written as that of 
finding a zero of $\boldsymbol{A}$ in \eqref{e:svva12f}. However, 
it is then reformulated in a new Hilbert space obtained by
suitably renorming $\HH\times\GG$. This
formulation yields an equivalent inclusion problem for an
operator which can be 
decomposed as the sum of two maximally monotone operators amenable
to forward-backward splitting (see Problem~\ref{prob:1} and
\eqref{e:fb1}) and, \emph{in fine}, an algorithm which 
also achieves full splitting (see \cite{Icip14,Cond13,Xiao12}
for related work). A special case of this framework is the 
algorithm proposed in \cite{Cham11}.
\item
\label{r:1iv}
The construction of algorithm~\eqref{e:blackpagepart3} revolves
around the problem of finding a zero of the operator 
$\boldsymbol{A}\in\ML(\HH\oplus\GG)\smallsetminus
\SL(\HH\oplus\GG)$ of \eqref{e:svva12f}.
It is not clear how this, or any of the splitting algorithms
mentioned in \ref{r:1i}--\ref{r:1iii}, could have
been devised using only subdifferential tools. 
\end{enumerate}
\end{remark}

\subsection{Lagrangian formulations of composite problems}

We consider a special case of Problem~\ref{prob:2} which
corresponds to the standard Fenchel-Rockafellar duality framework.

\begin{problem}
\label{prob:3}
Let $f\in\Gamma_0(\HH)$, let $\GG$ be a real Hilbert space, and
let $g\in\Gamma_0(\GG)$. Suppose that $0\neq L\in\BL(\HH,\GG)$
and that $0\in\ran(\partial f+L^*\circ\partial g\circ L)$. 
The objective is to solve the primal problem
\begin{equation}
\label{e:primalafr}
\minimize{x\in\HH}{f(x)+g(Lx)}
\end{equation}
as well as the dual problem
\begin{equation}
\label{e:dualafr}
\minimize{v\in\GG}{f^*(-L^*v)+g^*(v)}.
\end{equation}
\end{problem}

We have already discussed in Remark~\ref{r:1}\ref{r:1ii}
monotone operator-based algorithms to solve 
\eqref{e:primalafr}--\eqref{e:dualafr}.
Alternatively, set $\HHH=\HH\oplus\GG$, 
$\boldsymbol{f}\colon\HHH\to\RX\colon(x,y)\mapsto f(x)+g(y)$, and 
$\boldsymbol{L}\colon\HHH\to\GG\colon(x,y)\mapsto Lx-y$. 
Then \eqref{e:primalafr} is equivalent to
minimizing $\boldsymbol{f}$ over $\ker\boldsymbol{L}$. 
The Lagrangian for this type of
problem is $\mathcal{L}\colon\HHH\oplus\GG\to\RX\colon 
(\boldsymbol{x},v)\mapsto
\boldsymbol{f}(\boldsymbol{x})+\scal{\boldsymbol{Lx}}{v}$
\cite[Examp.~4']{Rock74} (see also \cite[Prop.~19.21]{Livre1}) 
and the associated maximally monotone
operator $\boldsymbol{A}$ of \eqref{e:rocky70} is defined 
at $(\boldsymbol{x},v)\in\HHH\oplus\GG$ to be 
$\boldsymbol{A}(\boldsymbol{x},v)=
(\partial\boldsymbol{f}(\boldsymbol{x})+\boldsymbol{L}^*v,
-\boldsymbol{Lx})$. Thus, solving 
\eqref{e:primalafr}--\eqref{e:dualafr}
is equivalent to finding a zero of the operator 
$\boldsymbol{A}\in\ML(\HH\oplus\GG\oplus\GG)\smallsetminus
\SL(\HH\oplus\GG\oplus\GG)$ defined by
\begin{equation}
\label{e:r74}
(\forall x\in\HH)(\forall y\in\GG)(\forall v\in\GG)\quad
\boldsymbol{A}(x,y,v)=
\big(\partial f(x)+L^*v,\partial g(y)-v,-Lx+y\big).
\end{equation}
In \cite{Ecks94}, this problem is approached by splitting
$\boldsymbol{A}$ as 
\begin{equation}
\label{e:r75}
(\forall x\in\HH)(\forall y\in\GG)(\forall v\in\GG)\quad
\boldsymbol{A}(x,y,v)=
\big(\partial f(x)+L^*v,0,-Lx\big)+
\big(0,\partial g(y)-v,y\big).
\end{equation}
Given $\gamma\in\RPP$, $\mu_1\in\RR$, $\mu_2\in\RR$, $x_0\in\HH$,
$y_0\in\GG$, and $v_0\in\GG$, applying the Douglas-Rachford 
algorithm \eqref{e:dr} to this decomposition leads to the algorithm
\cite{Ecks94}
\begin{equation}
\label{e:bp5}
\begin{array}{l}
\text{for}\;n=0,1,\ldots\\
\left\lfloor
\begin{array}{l}
x_{n+1}\in\Argmind{x\in\HH}{\bigg(f(x)+\scal{Lx}{v_n}+
\dfrac{1}{2\gamma}\|Lx-y_n\|^2+\dfrac{\gamma\mu_1^2}{2}
\|x-x_n\|^2\bigg)}\\[3mm]
y_{n+1}=\underset{y\in\GG}{\text{argmin}}\;\bigg(g(y)+\scal{y}{v_n}
+\dfrac{1}{2\gamma}\|Lx_{n+1}-y\|^2+\dfrac{\gamma\mu_2^2}{2}
\|y-y_n\|^2\bigg)\\[3mm]
v_{n+1}=v_n+\gamma^{-1}\big(Lx_{n+1}-y_{n+1}\big).
\end{array}
\right.\\
\end{array}
\end{equation}
When $\mu_1=\mu_2=0$, this scheme corresponds to the alternating 
direction method of multipliers (ADMM) 
\cite{Banf11,Glow89,Glow75,Glow16} and, just like it, requires a 
potentially complex minimization involving $f$ and $L$ jointly
to construct $x_{n+1}$ (see \cite{Ecks92,Gaba83} for connections
between ADMM and the Douglas-Rachford algorithm). To circumvent
this issue and obtain a method that does split $f$, $g$, and $L$,
let us decompose $\boldsymbol{A}$ as
$\boldsymbol{A}=\boldsymbol{M}+\boldsymbol{S}$, where
\begin{equation}
\label{e:r76}
(\forall x\in\HH)(\forall y\in\GG)(\forall v\in\GG)\quad
\begin{cases}
\boldsymbol{M}(x,y,v)=
\big(\partial f(x),\partial g(y),0\big)\\
\boldsymbol{S}(x,y,v)=\big(L^*v,-v,-Lx+y\big).
\end{cases}
\end{equation}
Applying \eqref{e:fbf} to this subdifferential+skew decomposition
in $\HH\oplus\GG\oplus\GG$, we obtain the following algorithm, 
which employs $\prox_f$, $\prox_g$, $L$, and $L^*$.

\begin{proposition}
\label{p:i}
Consider the setting of Problem~\ref{prob:3} and let
$(x_0,y_0,v_0)\in\HH\oplus\GG\oplus\GG$. Iterate
\begin{equation}
\label{e:blackpagepart5}
\begin{array}{l}
\text{for}\;n=0,1,\ldots\\
\left\lfloor
\begin{array}{l}
\varepsilon\leq\gamma_n\leq(1-\varepsilon)/\sqrt{1+\|L\|^2}\\
r_n=\gamma_n(Lx_n-y_n)\\
p_{n}=\prox_{\gamma_n f}\big(x_n-\gamma_nL^*v_{n}\big)\\[1mm]
q_{n}=\prox_{\gamma_n g}\big(y_{n}+\gamma_nv_n\big)\\[1mm]
x_{n+1}=p_n-\gamma_nL^*r_n\\[1mm]
y_{n+1}=q_n+\gamma_nr_n\\[1mm]
v_{n+1}=v_n+\gamma_n\big(Lp_n-q_n\big).
\end{array}
\right.\\
\end{array}
\end{equation}
Then $(x_n)_{n\in\NN}$ and $(v_n)_{n\in\NN}$ converge weakly to 
solutions to \eqref{e:primalafr} and \eqref{e:dualafr},
respectively.
\end{proposition}
\begin{proof}
This is an application of \cite[Thm.~2.5(ii)]{Siop11} to the
maximally monotone operator $\boldsymbol{M}$ and the monotone and 
Lipschitzian operator $\boldsymbol{S}$ of \eqref{e:r76}. Note that
the Lipschitz constant of $\boldsymbol{S}$ is 
$\|\boldsymbol{S}\|=\sqrt{1+\|L\|^2}$ and that
$(\forall n\in\NN)$ $J_{\gamma_n\boldsymbol{M}}=
\prox_{\gamma_n f}\times\prox_{\gamma_n g}\times\Id$. Thus, using
elementary algebraic manipulations, 
\eqref{e:fbf} reduces to \eqref{e:blackpagepart5}. 
\end{proof}

Let us note that \eqref{e:blackpagepart5} bears a certain
resemblance with the algorithm 
\begin{equation}
\label{e:ct94}
\begin{array}{l}
\text{for}\;n=0,1,\ldots\\
\left\lfloor
\begin{array}{l}
\varepsilon\leq\gamma_n\leq(1-\varepsilon)\text{min}\{1,1/\|L\|\}/2\\
p_{n}=v_n+\gamma_n(Lx_n-y_n)\\[1mm]
x_{n+1}=\prox_{\gamma_n f}\big(x_n-\gamma_nL^*p_{n}\big)\\[1mm]
y_{n+1}=\prox_{\gamma_n g}\big(y_{n}+\gamma_np_n\big)\\[1mm]
v_{n+1}=v_n+\gamma_n\big(Lx_{n+1}-y_{n+1}\big),
\end{array}
\right.\\
\end{array}
\end{equation}
proposed in a finite-dimensional setting in \cite{Chen94}. 

\section{Closing remarks}

The constant interactions between convex optimization and monotone 
operator theory have greatly benefited both fields. On
the numerical side, spectacular advances have been made in the last
years in the area of splitting algorithms to solve complex
structured problems. While many methods have been obtained by
recasting classical algorithms in product spaces, often with the
help of duality arguments, recent proposals such as that of
\cite{MaPr18} rely on different paradigms and make asynchronous and
block-iterative implementations possible. Despite the relative
maturity of the field, there remain
plenty of exciting open problems, and we can mention only a
few here. For instance, on the theoretical side, duality for monotone
inclusions is based on rather rudimentary principles, whereby dual
solutions exist if and only if primal solution exist, and it does
not match the more subtle results from Fenchel-Rockafellar duality
in classical convex optimization. On the algorithmic front,
splitting based on Bregman distances is still in its infancy.  This
framework is motivated by the need to solve problems in Banach
spaces, where standard notions of resolvent and proximity
operators
are no longer appropriate, but also by numerical considerations in
basic Euclidean spaces since some proximity operators may be easier
to implement in Bregman form or some functions may have more
exploitable properties when examined through Bregman
distances \cite{Jero17,Joca16,Nguy17}. As a final word, let us
emphasize that a monumental achievement of Browder, 
Ka{\v{c}}urovski{\u\i},
Minty, Moreau, Rockafellar, and Zarantonello was to build, within
the unchartered field of nonlinear analysis, structured
and fertile areas that extended ideas from classical linear
functional analysis. It remains a huge challenge to delimit and
construct such areas in the vast world of nonconvex/nonmonotone 
problems, that would preserve enough structure to support a 
solid and meaningful theory and, at the same time, lend itself 
to the development of powerful algorithms that would produce 
more than just local solutions.

\end{document}